\documentclass[a4paper]{amsart}

\usepackage{amscd}
\usepackage{amsmath}
\usepackage{amssymb}
\usepackage{amsthm}
\usepackage{bbm}

\usepackage[T1]{fontenc}
\usepackage[utf8]{inputenc}

\newif\ifpdf
\ifx\pdfoutput\undefined
   \pdffalse        % we are not running PDFLaTeX
\else
   \pdfoutput=1     % we are running PDFLaTeX
   \pdftrue
\fi

\ifpdf
   \usepackage[pdftex]{graphicx}
\else
   \usepackage{graphicx}
\fi

\numberwithin{equation}{section} \swapnumbers

\newtheorem{satz}{Satz}[section]

\newtheorem{theorem}[satz]{Theorem}
\newtheorem{proposition}[satz]{Proposition}
\newtheorem{corollary}[satz]{Corollary}
\newtheorem{lemma}[satz]{Lemma}

\newtheorem{definition}[satz]{Definition}

\newtheorem{remark}[satz]{Remark}

\begin{document}
\hyphenation{pre-sent cor-res-pond one di-men-sio-nal}

\title[Tempered stable distributions and processes]{Tempered stable distributions and processes}
\author{Uwe K{\"u}chler \and Stefan Tappe}
\address{Humboldt Universit\"{a}t zu Berlin, Institut f\"{u}r Mathematik, Unter den Linden 6, D-10099 Berlin, Germany}
\email{kuechler@mathematik.hu-berlin.de}
\address{Leibniz Universit\"{a}t Hannover, Institut f\"{u}r Mathematische Stochastik, Welfengarten 1, D-30167 Hannover, Germany}
\email{tappe@stochastik.uni-hannover.de}
\begin{abstract}
We investigate the class of tempered stable distributions and their associated processes. Our analysis of tempered stable distributions includes limit distributions, parameter estimation and the study of their densities. Regarding tempered stable processes, we deal with density transformations and compute their $p$-variation indices. Exponential stock models driven by tempered stable processes are discussed as well.
\end{abstract}
\maketitle

\section{Introduction}

Tempered stable distributions form a class of distributions that have attracted the interest of researchers from probability theory as well as financial mathematics. They have first been introduced in \cite{Koponen}, where the associated L\'{e}vy processes are called ``truncated L\'{e}vy flights'', and have been generalized by several authors. Tempered stable distributions form a six parameter family of infinitely divisible distributions, which cover several well-known subclasses like Variance Gamma distributions \cite{Madan-1990, Madan}, bilateral Gamma distributions \cite{Kuechler-Tappe} and CGMY distributions \cite{CGMY}. Properties of tempered stable distributions have been investigated, e.g., in \cite{Rosinski, Zhang,  Sztonyk, Bianchi-3}.
For financial modeling they have been applied, e.g., in \cite{Cont-Tankov, Mercuri, Bianchi-1, Bianchi-2}, see also the recent textbook \cite{Bianchi-book}.

The goal of the present paper is to contribute to the theory of tempered stable distributions and processes.
In detail, we provide limit results for tempered stable distributions, deal with statistical issues and analyze their density functions as well as path properties.

Tempered stable distributions cover the class of bilateral Gamma distributions, an analytical tractable class which we have investigated in \cite{Kuechler-Tappe, Kuechler-Tappe-shapes, Kuechler-Tappe-pricing}. Our subsequent investigations will show that, in many respects, the properties of bilateral Gamma distributions differ from those of all other tempered stable distributions (for example the properties of their densities, see Section~\ref{sec-densities}, or their $p$-variation indices, see Section~\ref{sec-p-variation}) and that bilateral Gamma distributions can be regarded as boundary points within the class of tempered stable distributions. In this paper, we are in particular interested in the question, which relevant properties for bilateral Gamma distributions still hold true for general tempered stable distributions.

The remainder of this text is organized as follows: In Section~\ref{sec-TS} we review tempered stable distributions and collect some basic properties. In Section~\ref{sec-closure} we investigate closure properties of tempered stable distributions with respect to weak convergence. Afterwards, in Section~\ref{sec-normal} we show convergence of tempered stable distributions to normal distributions and provide a convergence rate. In Section~\ref{sec-lln} we prove ``law of large numbers'' results, with a view to parameter estimation from observation of a typical sample path, and in Section~\ref{sec-statistics-distributions} we perform statistics for a finite number of realizations of tempered stable distributions. In Section~\ref{sec-densities} we analyze the densities of tempered stable distributions. In Section~\ref{sec-density-transformations} we investigate locally equivalent measures under which a tempered stable process remains tempered stable, and in Section~\ref{sec-p-variation} we compute the $p$-variation index of tempered stable processes. Finally, in Section~\ref{sec-finance} we present an application from mathematical finance and treat stock models driven by tempered stable processes.

\section{Tempered stable distributions and processes}\label{sec-TS}

In this section, we introduce tempered stable distributions and processes and collect their relevant properties.

We call an infinitely divisible distribution $\eta$ on $(\mathbb{R},\mathcal{B}(\mathbb{R}))$ a \emph{one-sided tempered stable distribution}, denoted $\eta = {\rm TS}(\alpha,\beta,\lambda)$, with parameters $\alpha,\lambda \in (0,\infty)$ and $\beta \in [0,1)$ if its characteristic function is given by
\begin{align}\label{cf-intro-one-sided}
\varphi(z) = \exp \bigg( \int_{\mathbb{R}} \big( e^{izx} - 1 \big) F(dx) \bigg), \quad z \in \mathbb{R} 
\end{align}
where the L\'{e}vy measure $F$ is
\begin{align}\label{Levy-measure-one-sided}
F(dx) = \frac{\alpha}{x^{1 + \beta}} e^{-\lambda x}
\mathbbm{1}_{(0,\infty)}(x) dx.
\end{align}
We call the L\'{e}vy process associated to $\eta$ a \emph{tempered stable subordinator}.

Next, we fix parameters $\alpha^+,\lambda^+,\alpha^-,\lambda^- \in (0,\infty)$ and $\beta^+,\beta^- \in [0,1)$. An infinitely divisible distribution $\eta$ on $(\mathbb{R},\mathcal{B}(\mathbb{R}))$ is called a \emph{tempered stable distribution}, denoted
\begin{align*}
\eta = {\rm TS}(\alpha^+,\beta^+,\lambda^+;\alpha^-,\beta^-,\lambda^-),
\end{align*}
if $\eta = \eta^+ * \eta^-$, where $\eta^+ = {\rm TS}(\alpha^+,\beta^+,\lambda^+)$ and $\eta^- = \widetilde{\nu}$ with $\nu = {\rm TS}(\alpha^-,\beta^-,\lambda^-)$ and $\widetilde{\nu}$ denoting the dual of $\nu$ given by $\widetilde{\nu}(B) = \nu(-B)$ for $B \in \mathcal{B}(\mathbb{R})$.
Note that $\eta$ has the characteristic function (\ref{cf-intro-one-sided})
with the L\'{e}vy measure $F$ given by
\begin{align}\label{Levy-measure-TS}
F(dx) = \left( \frac{\alpha^+}{x^{1 + \beta^+}} e^{-\lambda^+ x}
\mathbbm{1}_{(0,\infty)}(x) + \frac{\alpha^-}{|x|^{1 + \beta^-}}
e^{-\lambda^- |x|} \mathbbm{1}_{(-\infty,0)}(x) \right)dx.
\end{align}
We call the L\'{e}vy process associated to $\eta$ a \emph{tempered stable process}.

\begin{remark}
In \cite[Sec.~4.5]{Cont-Tankov}, the authors define generalized tempered stable processes for parameters $\alpha^+,\lambda^+,\alpha^-,\lambda^- > 0$ and $\beta^+,\beta^- < 2$ as L\'{e}vy processes with characteristic function
\begin{align}\label{cf-TS-inf-var}
\varphi(z) = \exp \bigg( iz\gamma + \int_{\mathbb{R}} \big( e^{izx} - 1 - izx \big) F(dx) \bigg), \quad z \in \mathbb{R} 
\end{align}
for some constant $\gamma \in \mathbb{R}$ and L\'{e}vy measure $F$ given by (\ref{Levy-measure-TS}). In the case  $\beta^+ = \beta^-$ they call such a process a tempered stable process.
The behaviour of the sample paths of a generalized tempered stable process $X$ depends on the values of $\beta^+, \beta^-$:
\begin{itemize}
\item For $\beta^+,\beta^- < 0$ we have $F(\mathbb{R}) < \infty$, and hence, $X$ is a compound Poisson process and of type A in the terminology of \cite[Def.~11.9]{Sato}.

\item For $\beta^+, \beta^- \in [0,1)$, which is the situation that we consider in the present paper, we have $F(\mathbb{R}) = \infty$, but $\int_{-1}^1 |x| F(dx) < \infty$. Therefore, $X$ is a finite-variation process making infinitely many jumps in each interval of positive length, which we can express as $X_t = \sum_{s \leq t} \Delta X_s$, and it belongs to type B in the terminology of \cite[Def.~11.9]{Sato}. In particular, we can decompose $X$ as the difference of two independent one-sided tempered stable subordinators.

\item For $\beta^+, \beta^- \in [1,2)$ we have $\int_{-1}^1 |x| F(dx) = \infty$. Therefore, the tempered stable process $X$ has sample paths of infinite variation and belongs to type C in the terminology of \cite[Def.~11.9]{Sato}.
\end{itemize}
\end{remark}

\begin{remark}
The tempered stable distributions considered in this paper also correspond to the generalized tempered stable distributions in \cite{Bianchi-book}. The following particular cases are known in the literature:
\begin{itemize}
\item $\beta^+ = \beta^-$ is a KoBol distribution, see \cite{Boy};

\item $\alpha^+ = \alpha^-$ and $\beta^+ = \beta^-$ is a CGMY-distribution, see \cite{CGMY}, also called classical tempered stable distribution in \cite{Bianchi-book};

\item $\beta^+ = \beta^-$ and $\lambda^+ = \lambda^-$ is the infinitely divisible distribution associated to a truncated L\'{e}vy flight, see \cite{Koponen};

\item $\beta^+ = \beta^- = 0$ is a bilateral Gamma distribution, see \cite{Kuechler-Tappe};

\item $\alpha^+ = \alpha^-$ and $\beta^+ = \beta^- = 0$ is a Variance Gamma distribution, see \cite{Madan-1990, Madan}.
\end{itemize}
According to \cite[Prop.~4.1]{Cont-Tankov}, a tempered stable process $X$ can be represented as a time changed Brownian motion with drift if and only if $X$ is a CGMY-process. Accordingly, a bilateral Gamma process is a time changed Brownian motion if and only if it is a Variance Gamma process.
\end{remark}

\begin{remark}
In \cite{Bianchi-book}, tempered stable distributions are considered as one-dimensional infinitely divisible distributions with L\'{e}vy measure
\begin{align}\label{tempering-q-F}
F(dx) = q(x) F_{\rm stable}(dx),
\end{align}
where
\begin{align*}
F_{\rm stable}(dx) = \left( \frac{\alpha^+}{x^{1 + \beta}} 
\mathbbm{1}_{(0,\infty)}(x) + \frac{\alpha^-}{|x|^{1 + \beta}}
\mathbbm{1}_{(-\infty,0)}(x) \right)dx
\end{align*}
is the L\'{e}vy measure of a $\beta$-stable distribution and $q : \mathbb{R} \rightarrow \mathbb{R}_+$ is a tempering function. For example, with
\begin{align*}
q(x) = e^{-\lambda^+ x} \mathbbm{1}_{(0,\infty)}(x) + e^{-\lambda^- |x|} \mathbbm{1}_{(-\infty,0)}(x)
\end{align*}
and $\alpha^+ = \alpha^-$ we have a CGMY-distribution. Further examples
are the modified tempered stable distribution, the normal tempered stable distribution, the Kim-Rachev tempered stable distribution and the rapidly decreasing tempered stable distribution, see \cite[Chapter~3.2]{Bianchi-book}. Note that the L\'{e}vy measures of the tempered stable distributions considered in our paper are generally not of the form (\ref{tempering-q-F}). 
\end{remark}

\begin{remark}
One can also consider multi-dimensional tempered stable distributions. In \cite{Rosinski}, a distribution $\eta$ on $(\mathbb{R}^d,\mathcal{B}(\mathbb{R}^d))$ is called tempered $\alpha$-stable for $\alpha \in (0,2)$ if it is infinitely divisible without Gaussian part and L\'{e}vy measure $M$, which in polar coordinates is of the form
\begin{align*}
M(dr,du) = \frac{q(r,u)}{r^{\alpha + 1}} dr \sigma(du),
\end{align*}
where $\sigma$ is a finite measure on the unit sphere $\mathbb{S}^{d-1}$ and $q : (0,\infty) \times \mathbb{S}^{d-1} \rightarrow (0,\infty)$ is a Borel function satisfying certain assumptions.
\end{remark}

We shall now collect some basic properties of generalized tempered stable distributions which we require for this text. In the sequel, $\Gamma : \mathbb{R} \setminus \{ 0,-1,-2,\ldots \} \rightarrow \mathbb{R}$ denotes the Gamma function.

\begin{lemma}\label{lemma-cf-pre}
Suppose that $\beta \in (0,1)$. The one-sided tempered stable distribution
\begin{align*}
\eta = {\rm TS}(\alpha,\beta,\lambda)
\end{align*}
has the characteristic function
\begin{align}\label{cf-TS-one-sided}
\varphi(z) &= \exp \Big( \alpha \Gamma(-\beta) \big[ (\lambda - iz)^{\beta} - \lambda^{\beta} \big] \Big), \quad z \in \mathbb{R}
\end{align}
where the power stems from the main branch of the complex logarithm.
\end{lemma}

\begin{proof}
Let $G \subset \mathbb{C}$ be the region $G = \{ z \in \mathbb{C} : {\rm Im} \, z > -\lambda \}$. We define the functions $f_i : G \rightarrow \mathbb{C}$ for $i=1,2$ as
\begin{align*}
f_1(z) := \int_{\mathbb{R}} \big( e^{izx} - 1 \big) F(dx) \quad \text{and} \quad f_2(z) := \alpha \Gamma(-\beta) \big[ (\lambda - iz)^{\beta} - \lambda^{\beta} \big].
\end{align*}
Then $f_1$ is analytic, which follows from \cite[Satz~IV.5.8]{Elstrodt}, and $f_2$ is analytic by the analyticity of the power function $z \mapsto z^{\beta}$ on the main branch of the complex logarithm. Let $B \subset G$ be the open ball $B = \{ z \in \mathbb{C} : |z| < \lambda \}$.
Using (\ref{Levy-measure-one-sided}) and Lebesgue's dominated convergence theorem, for all $z \in B$ we obtain
\begin{align*}
&f_1(z) = \int_{\mathbb{R}} \big( e^{izx} - 1 \big) F(dx) = \alpha \int_0^{\infty} \big( e^{izx} - 1 \big) \frac{e^{-\lambda x}}{x^{1 + \beta}} dx 
\\ &= \alpha \int_0^{\infty} \bigg( \sum_{n=1}^{\infty} \frac{(izx)^n}{n!} \bigg) \frac{e^{-\lambda x}}{x^{1 + \beta}} dx
= \alpha \sum_{n=1}^{\infty} \frac{(iz)^n}{n!} \int_0^{\infty} x^{n-\beta-1} e^{-\lambda x} dx 
\\ &= \alpha \sum_{n=1}^{\infty} \frac{(iz)^n}{n!} \int_0^{\infty} \bigg( \frac{x}{\lambda} \bigg)^{n-\beta-1} e^{-x} \frac{1}{\lambda} dx
= \alpha \sum_{n=1}^{\infty} \frac{(iz)^n}{n!} \lambda^{\beta - n} \Gamma(n-\beta) 
\\ &= \alpha \sum_{n=1}^{\infty} \frac{(iz)^n}{n!} \lambda^{\beta - n} \Gamma(-\beta) \prod_{k=0}^{n-1} (k-\beta)
= \alpha \Gamma(-\beta) \lambda^{\beta} \sum_{n=1}^{\infty} \bigg( \frac{iz}{\lambda} \bigg)^n (-1)^n \prod_{k=1}^n \frac{\beta - k + 1}{k} 
\\ &= \alpha \Gamma(-\beta) \lambda^{\beta} \sum_{n=1}^{\infty} \bigg( \genfrac{}{}{0pt}{}{\beta}{n} \bigg) \bigg( - \frac{iz}{\lambda} \bigg)^n
= \alpha \Gamma(-\beta) \lambda^{\beta} \bigg[ \bigg( 1 - \frac{iz}{\lambda} \bigg)^{\beta} - 1 \bigg] 
\\ &= \alpha \Gamma(-\beta) \big[ (\lambda - iz)^{\beta} - \lambda^{\beta} \big] = f_2(z).
\end{align*}
By the identity theorem for analytic functions we deduce that $f_1 \equiv f_2$ on $G$, which in particular yields that $f_1 \equiv f_2$ on $\mathbb{R}$. In view of (\ref{cf-intro-one-sided}), this proves (\ref{cf-TS-one-sided}).
\end{proof}

\begin{lemma}\label{lemma-cf}
Suppose that $\beta^+,\beta^- \in (0,1)$. The tempered stable distribution
\begin{align*}
\eta = {\rm TS}(\alpha^+,\beta^+,\lambda^+;\alpha^-,\beta^-,\lambda^-)
\end{align*}
has the characteristic function
\begin{equation}\label{cf-tempered-stable}
\begin{aligned}
\varphi(z) &= \exp \Big( \alpha^+ \Gamma(-\beta^+) \big[ (\lambda^+ - iz)^{\beta^+} - (\lambda^+)^{\beta^+} \big] 
\\ &\quad\quad\quad + \alpha^- \Gamma(-\beta^-) \big[ (\lambda^- + iz)^{\beta^-} - (\lambda^-)^{\beta^-} \big] \Big), \quad z \in \mathbb{R}
\end{aligned}
\end{equation}
where the powers stem from the main branch of the complex logarithm.
\end{lemma}

\begin{proof}
This is an immediate consequence of Lemma~\ref{lemma-cf-pre}.
\end{proof}

According to equation (2.2) in \cite{Kuechler-Tappe}, the characteristic function of a bilateral Gamma distribution (i.e. $\beta^+ = \beta^- = 0$) is given by
\begin{align}\label{cf-Gamma}
\varphi(z) = \bigg( \frac{\lambda^+}{\lambda^+ - iz} \bigg)^{\alpha^+} \bigg( \frac{\lambda^-}{\lambda^- + iz} \bigg)^{\alpha^-}, \quad z \in \mathbb{R}
\end{align}
where the powers stem from the main branch of the complex logarithm.

Using Lemma~\ref{lemma-cf}, for $\beta^+,\beta^- \in (0,1)$ the cumulant generating function 
\begin{align*}
\Psi(z) = \ln \mathbb{E} [e^{zX}] \quad \text{(where $X \sim {\rm TS}(\alpha^+,\beta^+,\lambda^+;\alpha^-,\beta^-,\lambda^-)$)}
\end{align*}
exists on $[-\lambda^-,\lambda^+]$ and is given by
\begin{equation}\label{cumulant}
\begin{aligned}
\Psi(z) &= \alpha^+ \Gamma(-\beta^+) \big[ (\lambda^+ - z)^{\beta^+} -
(\lambda^+)^{\beta^+} \big]
\\ &\quad + \alpha^- \Gamma(-\beta^-)
\big[ (\lambda^- + z)^{\beta^-} - (\lambda^-)^{\beta^-} \big], \quad z \in
[-\lambda^-,\lambda^+].
\end{aligned}
\end{equation}
For a bilateral Gamma distribution (i.e. $\beta^+ = \beta^- = 0$), the cumulant generating function exists on $(-\lambda^-,\lambda^+)$ and is given by
\begin{align}\label{Psi-Gamma}
\Psi(z) = \alpha^+ \ln \left( \frac{\lambda^+}{\lambda^+ - z}
\right) + \alpha^- \ln \left( \frac{\lambda^-}{\lambda^- + z} \right),
\quad z \in (-\lambda^-, \lambda^+)
\end{align}
see \cite[Sec.~2]{Kuechler-Tappe}.
Hence, for all $\beta^+,\beta^- \in [0,1)$ the $n$-th order cumulant $\kappa_n = \frac{d^n}{d z^n} \Psi(z) |_{z=0}$ is given by
\begin{align}\label{cumulant-kappa-gamma}
\kappa_n = \Gamma(n - \beta^+) \frac{\alpha^+}{(\lambda^+)^{n -
\beta^+}} + (-1)^n \Gamma(n-\beta^-) \frac{\alpha^-}{(\lambda^-)^{n
- \beta^-}}, \quad n \in \mathbb{N} = \{ 1,2,\ldots \}.
\end{align}
In particular, for a random variable
\begin{align*}
X \sim {\rm TS}(\alpha^+,\beta^+,\lambda^+;\alpha^-,\beta^-,\lambda^-)
\end{align*}
we can specify:
\begin{itemize}
\item The expectation
\begin{align}\label{expectation-TS}
\mathbb{E}[X] = \kappa_1 = \Gamma(1 - \beta^+) \frac{\alpha^+}{(\lambda^+)^{1 - \beta^+}} - \Gamma(1 - \beta^-) \frac{\alpha^-}{(\lambda^-)^{1 - \beta^-}}.
\end{align}

\item The variance
\begin{align}\label{variance-TS}
{\rm Var}[X] = \kappa_2 = \Gamma(2 - \beta^+) \frac{\alpha^+}{(\lambda^+)^{2 - \beta^+}} + \Gamma(2 - \beta^-) \frac{\alpha^-}{(\lambda^-)^{2 - \beta^-}}.
\end{align}

\item The Charliers skewness
\begin{align}
\gamma_1(X) = \frac{\kappa_3}{\kappa_2^{3/2}} = \frac{\Gamma(3 - \beta^+) \frac{\alpha^+}{(\lambda^+)^{3 - \beta^+}} - \Gamma(3 - \beta^-) \frac{\alpha^-}{(\lambda^-)^{3 - \beta^-}}}{\left(
\Gamma(2 - \beta^+) \frac{\alpha^+}{(\lambda^+)^{2 - \beta^+}} + \Gamma(2 - \beta^-) \frac{\alpha^-}{(\lambda^-)^{2 - \beta^-}}
\right)^{3/2}}.
\end{align}

\item The kurtosis
\begin{align}\label{kurtosis}
\gamma_2(X) = 3 + \frac{\kappa_4}{\kappa_2^2} = 3 + \frac{\Gamma(4 - \beta^+) \frac{\alpha^+}{(\lambda^+)^{4 - \beta^+}} + \Gamma(4 - \beta^-) \frac{\alpha^-}{(\lambda^-)^{4 - \beta^-}}}{\left(
\Gamma(2 - \beta^+) \frac{\alpha^+}{(\lambda^+)^{2 - \beta^+}} + \Gamma(2 - \beta^-) \frac{\alpha^-}{(\lambda^-)^{2 - \beta^-}}
\right)^2}.
\end{align}
\end{itemize}

\begin{remark}
Let $\eta = {\rm TS}(\alpha,\beta,\lambda)$ be a one-sided tempered stable distribution. For $\beta \in (0,1)$ the characteristic function is given by (\ref{cf-TS-one-sided}), see Lemma~\ref{lemma-cf-pre},
and hence, the cumulant generating function exists on $(-\infty,\lambda]$ and is given by
\begin{align}\label{cumulant-TS-one-sided}
\Psi(z) = \alpha \Gamma(-\beta) \big[ (\lambda - z)^{\beta} - \lambda^{\beta} \big], \quad z \in (-\infty,\lambda].
\end{align}
For $\beta = 0$ the tempered stable distribution is a Gamma distribution $\eta = \Gamma(\alpha,\lambda)$. Therefore, we have the characteristic function
\begin{align}\label{cf-Gamma-one-sided}
\varphi(z) = \bigg( \frac{\lambda}{\lambda - iz} \bigg)^{\alpha}, \quad z \in \mathbb{R}
\end{align}
and the cumulant generating function exists on $(-\infty,\lambda)$ and is given by
\begin{align}\label{cumulent-Gamma-one-sided}
\Psi(z) = \alpha \ln \bigg( \frac{\lambda}{\lambda - z} \bigg), \quad z \in (-\infty,\lambda).
\end{align}
For $\beta \in [0,1)$ and $X \sim {\rm TS}(\alpha,\beta,\lambda)$ we obtain 
the cumulants
\begin{align}\label{cumulant-one-sided}
\kappa_n = \Gamma(n-\beta) \frac{\alpha}{\lambda^{n-\beta}}, \quad n \in \mathbb{N},
\end{align}
the expectation and the variance
\begin{align}\label{E-TS-one-sided}
\mathbb{E}[X] &= \Gamma(1 - \beta) \frac{\alpha}{\lambda^{1 - \beta}},
\\ \label{Var-TS-one-sided} {\rm Var}[X] &= \Gamma(2 - \beta) \frac{\alpha}{\lambda^{2 - \beta}}.
\end{align}
\end{remark}

\begin{remark}\label{remark-cumulant-inf-var}
The characteristic function (\ref{cf-TS-inf-var}) of the tempered stable distribution
\begin{align*}
\eta = {\rm TS}(\alpha^+,\beta^+,\lambda^+;\alpha^-,\beta^-,\lambda^-)
\end{align*}
with $\beta^+,\beta^- < 2$ is given by
\begin{align*}
\varphi(z) &= \exp \Big( iz\gamma + \alpha^+ \Gamma(-\beta^+) \big[ (\lambda^+ - iz)^{\beta^+} - (\lambda^+)^{\beta^+} + iz \beta^+ (\lambda^+)^{\beta^+ - 1} \big] 
\\ &\quad\quad\quad + \alpha^- \Gamma(-\beta^-) \big[ (\lambda^- + iz)^{\beta^-} - (\lambda^-)^{\beta^-} - iz \beta^- (\lambda^-)^{\beta^- - 1} \big] \Big), \quad z \in \mathbb{R}.
\end{align*}
Therefore, the cumulant generating function exists on $[-\lambda^-,\lambda^+]$ and is given by
\begin{align*}
\Psi(z) &= \gamma z + \alpha^+ \Gamma(-\beta^+) \big[ (\lambda^+ - z)^{\beta^+} -
(\lambda^+)^{\beta^+} + \beta^+ (\lambda^+)^{\beta^+ - 1} z \big]
\\ &\quad + \alpha^- \Gamma(-\beta^-)
\big[ (\lambda^- + z)^{\beta^-} - (\lambda^-)^{\beta^-} - \beta^- (\lambda^-)^{\beta^- - 1} z \big], \quad z \in
[-\lambda^-,\lambda^+].
\end{align*}
Hence, we have $\kappa_1 = \gamma$ and the cumulants $\kappa_n$ for $n \geq 2$ are given by (\ref{cumulant-kappa-gamma}).
\end{remark}

For a tempered stable process $X$ we shall also write 
\begin{align*}
X \sim {\rm TS}(\alpha^+,\beta^+,\lambda^+;\alpha^-,\beta^-,\lambda^-).
\end{align*}
Note that we can decompose $X = X^+ - X^-$ as the difference of two independent one-sided tempered stable subordinators $X^+ \sim {\rm TS}(\alpha^+,\beta^+,\lambda^+)$ and $X^- \sim {\rm TS}(\alpha^-,\beta^-,\lambda^-)$. In view of the characteristic function (\ref{cf-tempered-stable}) calculated in Lemma~\ref{lemma-cf}, all increments of $X$ have a tempered stable distribution, more precisely
\begin{align}\label{distribution-TS-time}
X_t - X_s \sim {\rm TS}(\alpha^+ (t-s), \beta^+, \lambda^+; \alpha^- (t-s), \beta^-,
\lambda^-) \quad \text{for $0 \leq s < t$.}
\end{align}
In particular, for any constant $\Delta > 0$ the process $X_{\Delta \bullet} = (X_{\Delta t})_{t \geq 0}$ is a tempered stable process
\begin{align}\label{X-Delta-bullet}
X_{\Delta \bullet} \sim {\rm TS}(\Delta \alpha^+,\beta^+,\lambda^+;\Delta \alpha^-,\beta^-,\lambda^-).
\end{align}

\section{Closure properties of tempered stable distributions}\label{sec-closure}

In this section, we shall investigate limit distributions of sequences of tempered stable distributions.

\begin{proposition}\label{prop-closure}
Let sequences
\begin{align*}
(\alpha_n^+,\beta_n^+,\lambda_n^+;\alpha_n^-,\beta_n^-,\lambda_n^-)_{n \in \mathbb{N}} \subset ( (0,\infty) \times [0,1) \times (0,\infty) )^2
\end{align*}
and real numbers
\begin{align*}
(\alpha^+,\beta^+,\lambda^+;\alpha^-,\beta^-,\lambda^-) \in ( (0,\infty) \times [0,1) \times (0,\infty) )^2
\end{align*}
be given. Then, the following statements are valid:
\begin{enumerate}
\item If we have
\begin{align*}
(\alpha_n^+,\beta_n^+,\lambda_n^+;\alpha_n^-,\beta_n^-,\lambda_n^-) \rightarrow (\alpha^+,\beta^+,\lambda^+;\alpha^-,\beta^-,\lambda^-),
\end{align*}
then we have the weak convergence
\begin{align}\label{weak-cont}
{\rm TS}(\alpha_n^+,\beta_n^+,\lambda_n^+;\alpha_n^-,\beta_n^-,\lambda_n^-) \overset{w}{\rightarrow} {\rm TS}(\alpha^+,\beta^+,\lambda^+;\alpha^-,\beta^-,\lambda^-).
\end{align}

\item If we have $\alpha_n^- \rightarrow 0$ and
\begin{align*}
(\alpha_n^+,\beta_n^+,\lambda_n^+;\beta_n^-,\lambda_n^-) \rightarrow (\alpha^+,\beta^+,\lambda^+;\beta^-,\lambda^-),
\end{align*}
then we have the weak convergence
\begin{align}\label{weak-Dirac-1}
{\rm TS}(\alpha_n^+,\beta_n^+,\lambda_n^+;\alpha_n^-,\beta_n^-,\lambda_n^-) \overset{w}{\rightarrow} {\rm TS}(\alpha^+,\beta^+,\lambda^+).
\end{align}

\item If we have $\alpha_n^+,\alpha_n^- \rightarrow 0$ and
\begin{align*}
(\beta_n^+,\lambda_n^+;\beta_n^-,\lambda_n^-) \rightarrow (\beta^+,\lambda^+;\beta^-,\lambda^-),
\end{align*}
then we have the weak convergence
\begin{align}\label{weak-Dirac-2}
{\rm TS}(\alpha_n^+,\beta_n^+,\lambda_n^+;\alpha_n^-,\beta_n^-,\lambda_n^-) \overset{w}{\rightarrow} \delta_0.
\end{align}

\item If we have $\alpha_n^+,\alpha_n^-,\lambda_n^+,\lambda_n^- \rightarrow \infty$ and there are $\mu^+,\mu^- \geq 0$ such that
\begin{align}\label{LLN-cond}
\Gamma(1-\beta^+) \frac{\alpha_n^+}{(\lambda_n^+)^{1-\beta^+}} \rightarrow \mu^+ \quad \text{and} \quad \Gamma(1-\beta^-) \frac{\alpha_n^-}{(\lambda_n^-)^{1-\beta^-}} \rightarrow \mu^-,
\end{align}
then we have the weak convergence
\begin{align}\label{weak-LLN}
{\rm TS}(\alpha_n^+,\beta^+,\lambda_n^+;\alpha_n^-,\beta^-,\lambda_n^-) \overset{w}{\rightarrow} \delta_{\mu},
\end{align}
where the number $\mu \in \mathbb{R}$ is the limit of the means 
\begin{align*}
\mu = \mu^+ - \mu^-.
\end{align*}

\item If we have $\alpha_n^+,\alpha_n^-,\lambda_n^+,\lambda_n^- \rightarrow \infty$ and there are $\mu \in \mathbb{R}$, $(\sigma^+)^2, (\sigma^-)^2 > 0$ such that
\begin{align}\label{CLT-cond-1}
&\Gamma(1-\beta^+) \frac{\alpha_n^+}{(\lambda_n^+)^{1-\beta^+}} - \Gamma(1-\beta^-) \frac{\alpha_n^-}{(\lambda_n^-)^{1-\beta^-}} \rightarrow \mu,
\\ \label{CLT-cond-2} &\Gamma(2-\beta^+) \frac{\alpha_n^+}{(\lambda_n^+)^{2-\beta^+}} \rightarrow (\sigma^+)^2  \quad \text{and} \quad \Gamma(2-\beta^-) \frac{\alpha_n^-}{(\lambda_n^-)^{2-\beta^-}} \rightarrow (\sigma^-)^2, 
\end{align}
then we have the weak convergence
\begin{align}\label{weak-normal}
{\rm TS}(\alpha_n^+,\beta^+,\lambda_n^+;\alpha_n^-,\beta^-,\lambda_n^-) \overset{w}{\rightarrow} N(\mu,\sigma^2),
\end{align}
where the variance $\sigma^2 > 0$ is given by
\begin{align*}
\sigma^2 = (\sigma^+)^2 + (\sigma^-)^2.
\end{align*}
\end{enumerate}
\end{proposition}

\begin{proof}
In order to prove the result, it suffices to show that the respective characteristic functions converge. Noting that, by l'H\^{o}pital's rule, for all $\lambda > 0$ and $z \in \mathbb{R}$ we have
\begin{align*}
&\lim_{\beta \downarrow 0} \Gamma(-\beta) \big[ (\lambda - iz)^{\beta} - \lambda^{\beta} \big] = \lim_{\beta \downarrow 0} \Gamma(-\beta) \lambda^{\beta} \bigg[ \bigg( \frac{\lambda - iz}{\lambda} \bigg)^{\beta} - 1 \bigg] 
\\ &= - \lim_{\beta \downarrow 0} \Gamma(1-\beta) \lambda^{\beta} \frac{\big(\frac{\lambda - iz}{\lambda}\big)^{\beta} - 1}{\beta} = - \lim_{\beta \downarrow 0} \frac{\frac{d}{d \beta} \big( \big(\frac{\lambda - iz}{\lambda}\big)^{\beta} - 1 \big)}{\frac{d}{d \beta} \beta}
\\ &= - \lim_{\beta \downarrow 0} \bigg( \frac{\lambda - iz}{\lambda} \bigg)^{\beta} \ln \bigg( \frac{\lambda - iz}{\lambda} \bigg) = \ln \bigg( \frac{\lambda}{\lambda - iz} \bigg),
\end{align*}
relations (\ref{weak-cont})--(\ref{weak-Dirac-2}) follow by taking into account the representations (\ref{cf-tempered-stable}), (\ref{cf-Gamma}) and (\ref{cf-TS-one-sided}), (\ref{cf-Gamma-one-sided}) of the characteristic functions.

In the sequel, for $n \in \mathbb{N}$ we denote by $\varphi_n : \mathbb{R} \rightarrow \mathbb{C}$ the characteristic function of the tempered stable distribution
\begin{align*}
{\rm TS}(\alpha_n^+,\beta^+,\lambda_n^+;\alpha_n^-,\beta^-,\lambda_n^-)
\end{align*}
and by $(\kappa_j^n)_{j \in \mathbb{N}}$ its cumulants. Then we have
\begin{align*}
\varphi_n(u) = \exp \bigg( \sum_{j=1}^{\infty} \frac{(iu)^j}{j!} \kappa_j^n \bigg), \quad u \in \mathbb{R}.
\end{align*}
Suppose that (\ref{LLN-cond}) is satisfied. Using the estimates
\begin{align}\label{est-Gamma-faculty}
\Gamma(j-\beta^+) \leq (j-1)! \quad \text{and} \quad \Gamma(j-\beta^-) \leq (j-1)! \quad \text{for $j \geq 2$,}
\end{align}
by the geometric series and (\ref{LLN-cond}), for all $u \in \mathbb{R}$ we obtain
\begin{align*}
&\sum_{j=2}^{\infty} \bigg| \frac{(iu)^j}{j!} \bigg(  \Gamma(j - \beta^+) \frac{\alpha_n^+}{(\lambda_n^+)^{j -
\beta^+}} + (-1)^j \Gamma(j-\beta^-) \frac{\alpha_n^-}{(\lambda_n^-)^{j
- \beta^-}} \bigg) \bigg| 
\\ &\leq \frac{\alpha_n^+}{(\lambda_n^+)^{1 -
\beta^+}} |u| \sum_{j=2}^{\infty} \bigg( \frac{|u|}{\lambda_n^+} \bigg)^{j-1} + \frac{\alpha_n^-}{(\lambda_n^-)^{1 -
\beta^-}} |u| \sum_{j=2}^{\infty} \bigg( \frac{|u|}{\lambda_n^-} \bigg)^{j-1}
\\ &= \frac{\alpha_n^+}{(\lambda_n^+)^{1 -
\beta^+}} \frac{|u|^2}{\lambda_n^+ - |u|} + \frac{\alpha_n^-}{(\lambda_n^-)^{1 - \beta^-}} \frac{|u|^2}{\lambda_n^- - |u|} \rightarrow 0 \quad \text{as $n \rightarrow \infty$,}
\end{align*}
and hence, by taking into account (\ref{cumulant-kappa-gamma}) and (\ref{LLN-cond}), we have
\begin{align*}
\varphi_n(u) \rightarrow e^{iu\mu}, \quad u \in \mathbb{R}
\end{align*}
proving (\ref{weak-LLN}). Now, suppose that (\ref{CLT-cond-1}), (\ref{CLT-cond-2}) are satisfied. Using the estimates (\ref{est-Gamma-faculty}),
by the geometric series and (\ref{CLT-cond-2}), for all $u \in \mathbb{R}$ we obtain
\begin{align*}
&\sum_{j=3}^{\infty} \bigg| \frac{(iu)^j}{j!} \bigg(  \Gamma(j - \beta^+) \frac{\alpha_n^+}{(\lambda_n^+)^{j -
\beta^+}} + (-1)^j \Gamma(j-\beta^-) \frac{\alpha_n^-}{(\lambda_n^-)^{j
- \beta^-}} \bigg) \bigg| 
\\ &\leq \frac{\alpha_n^+}{(\lambda_n^+)^{2 -
\beta^+}} |u|^2 \sum_{j=3}^{\infty} \bigg( \frac{|u|}{\lambda_n^+} \bigg)^{j-2} + \frac{\alpha_n^-}{(\lambda_n^-)^{2 -
\beta^-}} |u|^2 \sum_{j=3}^{\infty} \bigg( \frac{|u|}{\lambda_n^-} \bigg)^{j-2}
\\ &= \frac{\alpha_n^+}{(\lambda_n^+)^{2 -
\beta^+}} \frac{|u|^3}{\lambda_n^+ - |u|} + \frac{\alpha_n^-}{(\lambda_n^-)^{2 - \beta^-}} \frac{|u|^3}{\lambda_n^- - |u|} \rightarrow 0 \quad \text{as $n \rightarrow \infty$,}
\end{align*}
and hence, by taking into account (\ref{cumulant-kappa-gamma}) and (\ref{CLT-cond-1}), (\ref{CLT-cond-2}), we have
\begin{align*}
\varphi_n(u) \rightarrow e^{iu\mu - u^2 \sigma^2 / 2}, \quad u \in \mathbb{R}
\end{align*}
proving (\ref{weak-normal}).
\end{proof}

By (\ref{weak-cont}), the class of tempered stable distributions is closed under weak convergence on its domain
\begin{align*}
( (0,\infty) \times [0,1) \times (0,\infty) )^2.
\end{align*}
Note that bilateral Gamma distributions (corresponding to $\beta^+ = 0$ and $\beta^- = 0$) belong to this domain and are contained in its boundary.
For $\alpha^+ \rightarrow 0$ or $\alpha^- \rightarrow 0$ we obtain one-sided tempered stable distributions and Dirac measures as boundary distributions, see (\ref{weak-Dirac-1}) and (\ref{weak-Dirac-2}). If $\alpha^+,\alpha^-,\lambda^+,\lambda^- \rightarrow \infty$ for fixed valued of $\beta^+,\beta^-$, in certain situations we obtain a Dirac measure, see (\ref{weak-LLN}), or to a normal distribution, see (\ref{weak-normal}), as limit distribution.

\section{Convergence of tempered stable distributions to a normal distribution}\label{sec-normal}

In \cite[Sec.~3]{Rosinski} the long time behaviour of tempered stable processes was studied and convergence to a Brownian motion was established. Here, we provide a convergence rate and show, how close a given tempered stable distribution (or tempered stable process) is to a normal distribution (or Brownian motion).

\begin{lemma}\label{lemma-TS-consistent}
The following statements are valid:
\begin{enumerate}
\item Suppose $X_i \sim {\rm TS}(\alpha_i^+,\beta^+,\lambda^+;\alpha_i^-,\beta^-,\lambda^-)$, $i=1,2$ are independent. Then we have
\begin{align}\label{TS-sum}
X_1 + X_2 \sim {\rm TS}(\alpha_1^+ + \alpha_2^+,\beta^+,\lambda^+;\alpha_1^- + \alpha_2^-,\beta^-,\lambda^-).
\end{align}
\item For $X \sim {\rm TS}(\alpha^+,\beta^+,\lambda^+;\alpha^-,\beta^-,\lambda^-)$ and a constant $\rho > 0$ we have
\begin{align}\label{TS-scalar}
\rho X \sim {\rm TS}(\alpha^+ {\rho}^{\beta^+},\beta^+,\lambda^+ / \rho;\alpha^- {\rho}^{\beta^-},\beta^-,\lambda^- / \rho).
\end{align}
\end{enumerate}
\end{lemma}

\begin{proof}
For independent $X_i \sim {\rm TS}(\alpha_i^+,\beta^+,\lambda^+;\alpha_i^-,\beta^-,\lambda^-)$, $i=1,2$ we have by Lemma~\ref{lemma-cf} the characteristic function
\begin{align*}
&\varphi_{X_1 + X_2}(z) = \varphi_{X_1}(z) \varphi_{X_2}(z) 
\\ &= \exp \Big( \alpha_1^+ \Gamma(-\beta^+) \big[ (\lambda^+ - iz)^{\beta^+} - (\lambda^+)^{\beta^+} \big] 
+ \alpha_1^- \Gamma(-\beta^-) \big[ (\lambda^- + iz)^{\beta^-} - (\lambda^-)^{\beta^-} \big] \Big)
\\ &\times \exp \Big( \alpha_2^+ \Gamma(-\beta^+) \big[ (\lambda^+ - iz)^{\beta^+} - (\lambda^+)^{\beta^+} \big] 
+ \alpha_2^- \Gamma(-\beta^-) \big[ (\lambda^- + iz)^{\beta^-} - (\lambda^-)^{\beta^-} \big] \Big)
\\ &= \exp \Big( (\alpha_1^+ + \alpha_2^+) \Gamma(-\beta^+) \big[ (\lambda^+ - iz)^{\beta^+} - (\lambda^+)^{\beta^+} \big] 
\\ &\quad + (\alpha_1^- + \alpha_2^-) \Gamma(-\beta^-) \big[ (\lambda^- + iz)^{\beta^-} - (\lambda^-)^{\beta^-} \big] \Big),
\end{align*}
showing (\ref{TS-sum}). For $X \sim {\rm TS}(\alpha^+,\beta^+,\lambda^+;\alpha^-,\beta^-,\lambda^-)$ and $\rho > 0$ we have by Lemma~\ref{lemma-cf} the characteristic function
\begin{align*}
\varphi_{\rho X}(z) = \varphi_X(\rho z) 
&= \exp \Big( \alpha^+ \Gamma(-\beta^+) \big[ (\lambda^+ - i \rho z)^{\beta^+} - (\lambda^+)^{\beta^+} \big] 
\\ &\qquad\qquad + \alpha^- \Gamma(-\beta^-) \big[ (\lambda^- + i \rho z)^{\beta^-} - (\lambda^-)^{\beta^-} \big] \Big)
\\ &= \exp \Big( \alpha^+ {\rho}^{\beta^+} \Gamma(-\beta^+) \big[ (\lambda^+ / \rho - iz)^{\beta^+} - (\lambda^+ / \rho)^{\beta^+} \big] 
\\ &\qquad\qquad + \alpha^- {\rho}^{\beta^-} \Gamma(-\beta^-) \big[ (\lambda^- / \rho + iz)^{\beta^-} - (\lambda^- / \rho)^{\beta^-} \big] \Big),
\end{align*}
showing (\ref{TS-scalar}).
\end{proof}

\begin{lemma}\label{lemma-transform-para}
Let $X$ be a random variable
\begin{align*}
X &\sim {\rm TS}(\alpha^+,\beta^+,\lambda^+;\alpha^-,\beta^-,\lambda^-). 
\end{align*}
We set $\mu := \mathbb{E}[X]$ and $\sigma^2 := {\rm Var}[X]$. Let $\rho,\tau > 0$ and
\begin{align*}
Y &\sim {\rm TS}(\rho \alpha^+/(\tau \sqrt{\rho})^{\beta^+},\beta^+,\lambda^+ \tau \sqrt{\rho}; \rho \alpha^-/(\tau \sqrt{\rho})^{\beta^-},\beta^-,\lambda^- \tau \sqrt{\rho}).
\end{align*}
Then we have 
\begin{align*}
\mathbb{E}[Y] = \frac{\sqrt{\rho}}{\tau} \mu \quad \text{and} \quad {\rm Var}[Y] = \frac{\sigma^2}{\tau^2}.
\end{align*}
\end{lemma}

\begin{proof}
Note that by (\ref{expectation-TS}), (\ref{variance-TS}) we have
\begin{align*}
\mu &= \Gamma(1 - \beta^+) \frac{\alpha^+}{(\lambda^+)^{1 - \beta^+}} - \Gamma(1 - \beta^-) \frac{\alpha^-}{(\lambda^-)^{1 - \beta^-}},
\\ \sigma^2 &= \Gamma(2 - \beta^+) \frac{\alpha^+}{(\lambda^+)^{2 - \beta^+}} + \Gamma(2 - \beta^-) \frac{\alpha^-}{(\lambda^-)^{2 - \beta^-}}.
\end{align*}
Therefore, we obtain
\begin{align*}
\mathbb{E}[Y] &= \Gamma(1 - \beta^+) \frac{\rho \alpha^+}{(\tau \sqrt{\rho})^{\beta^+} (\lambda^+ \tau \sqrt{\rho})^{1-\beta^+}} - \Gamma(1 - \beta^-) \frac{\rho \alpha^-}{(\tau \sqrt{\rho})^{\beta^-} (\lambda^- \tau \sqrt{\rho})^{1-\beta^-}}
\\ &= \frac{\sqrt{\rho}}{\tau} \bigg( \Gamma(1 - \beta^+) \frac{\alpha^+}{(\lambda^+)^{1 - \beta^+}} - \Gamma(1 - \beta^-) \frac{\alpha^-}{(\lambda^-)^{1 - \beta^-}} \bigg) = \frac{\sqrt{\rho}}{\tau} \mu
\end{align*}
as well as
\begin{align*}
{\rm Var}[Y] &= \Gamma(2 - \beta^+) \frac{\rho \alpha^+}{(\tau \sqrt{\rho})^{\beta^+} (\lambda^+ \tau \sqrt{\rho})^{2-\beta^+}} + \Gamma(2 - \beta^-) \frac{\rho \alpha^-}{(\tau \sqrt{\rho})^{\beta^-} (\lambda^- \tau \sqrt{\rho})^{2-\beta^-}}
\\ &= \frac{1}{\tau^2} \bigg( \Gamma(2 - \beta^+) \frac{\alpha^+}{(\lambda^+)^{2 - \beta^+}} + \Gamma(2 - \beta^-) \frac{\alpha^-}{(\lambda^-)^{2 - \beta^-}} \bigg) = \frac{\sigma^2}{\tau^2},
\end{align*}
finishing the proof.
\end{proof}

\begin{lemma}\label{lemma-det-alpha}
Let $X$ be a random variable
\begin{align*}
X \sim {\rm TS}(\alpha^+,\beta^+,\lambda^+;\alpha^-,\beta^-,\lambda^-)
\end{align*}
and let $\mu \in \mathbb{R}$ and $\sigma^2 > 0$ be arbitrary. The following statements are equivalent:
\begin{enumerate}
\item We have $\mathbb{E}[X] = \mu$ and ${\rm Var}[X] = \sigma^2$. 

\item We have
\begin{align}\label{sol-alpha-1}
\alpha^+ &= \frac{(\lambda^+)^{2-\beta^+} ((1-\beta^-)\mu + \lambda^- \sigma^2)}{\Gamma(1-\beta^+)((1-\beta^-)\lambda^+ + (1-\beta^+)\lambda^-)},
\\ \label{sol-alpha-2} \alpha^- &= \frac{(\lambda^-)^{2-\beta^-} ((\beta^+ - 1)\mu + \lambda^+ \sigma^2)}{\Gamma(1-\beta^-)((1-\beta^-)\lambda^+ + (1-\beta^+)\lambda^-)}.
\end{align}
\end{enumerate}
\end{lemma}

\begin{proof}
Let $A \in \mathbb{R}^{2 \times 2}$ be the matrix
\begin{align*}
A = \left(
\begin{array}{cc}
a_{11} & a_{12}
\\ a_{21} & a_{22}
\end{array}
\right)
= \left(
\begin{array}{cc}
\Gamma(1-\beta^+)/(\lambda^+)^{1-\beta^+} & -\Gamma(1-\beta^-)/(\lambda^-)^{1-\beta^-}
\\ \Gamma(2-\beta^+)/(\lambda^+)^{2-\beta^+} & \Gamma(2-\beta^-)/(\lambda^+)^{2-\beta^-}
\end{array}
\right).
\end{align*}
Then we have
\begin{equation}\label{det-LGS}
\begin{aligned}
\det A &= \frac{\Gamma(1-\beta^+) \Gamma(2-\beta^-)}{(\lambda^+)^{1-\beta^+} (\lambda^-)^{2-\beta^-}} + \frac{\Gamma(2-\beta^+) \Gamma(1-\beta^-)}{(\lambda^+)^{2-\beta^+} (\lambda^-)^{1-\beta^-}}
\\ &= \frac{\Gamma(1-\beta^+) \Gamma(1-\beta^-)}{(\lambda^+)^{2-\beta^+} (\lambda^-)^{2-\beta^-}} \Big( \lambda^+ (1-\beta^-) + \lambda^- (1-\beta^+) \Big) > 0.
\end{aligned}
\end{equation}
Using (\ref{expectation-TS}), (\ref{variance-TS}), a straightforward calculation shows that
\begin{align}\label{matrix-1}
a_{11} \alpha^+ &= \lambda^+ \frac{(1-\beta^-)\mu + \lambda^- \sigma^2}{(1-\beta^-) \lambda^+ + (1-\beta^+) \lambda^-},
\\ a_{12} \alpha^- &= -\lambda^- \frac{-(1-\beta^+)\mu + \lambda^+ \sigma^2}{(1-\beta^-) \lambda^+ + (1-\beta^+) \lambda^-},
\\ a_{21} \alpha^+ &= (1-\beta^+) \frac{(1-\beta^-)\mu + \lambda^- \sigma^2}{(1-\beta^-) \lambda^+ + (1-\beta^+) \lambda^-},
\\ \label{matrix-4} a_{22} \alpha^- &= (1-\beta^-) \frac{-(1-\beta^+)\mu + \lambda^+ \sigma^2}{(1-\beta^-) \lambda^+ + (1-\beta^+) \lambda^-}.
\end{align}
By (\ref{expectation-TS}), (\ref{variance-TS}), we have $\mathbb{E}[X] = \mu$ and ${\rm Var}[X] = \sigma^2$ if and only if
\begin{align}\label{LGS}
A \cdot \left( 
\begin{array}{c} 
\alpha^+
\\ \alpha^-
\end{array}
\right) = \left( 
\begin{array}{c} 
\mu
\\ \sigma^2
\end{array}
\right).
\end{align}
Because of (\ref{det-LGS}), the system of linear equations (\ref{LGS}) has a unique solution. Taking into account (\ref{matrix-1})--(\ref{matrix-4}), the solution for (\ref{LGS}) is given by (\ref{sol-alpha-1}), (\ref{sol-alpha-2}).
\end{proof}

\begin{lemma}\label{lemma-comp-single-TS}
For $X \sim {\rm TS}(\alpha,\beta,\lambda)$ we have
\begin{align*}
\mathbb{E}[X^3] = \Gamma(1-\beta) \frac{\alpha}{\lambda^{3-\beta}} \big[ \Gamma(1-\beta)^2 \alpha^2 \lambda^{2\beta} + 3 (1-\beta) \Gamma(1-\beta) \alpha \lambda^{\beta} + (1-\beta)(2-\beta) \big].
\end{align*}
\end{lemma}

\begin{proof}
Denoting by $\kappa_1, \kappa_2, \kappa_3$ the first three cumulants of $X$, by \cite[p.~346]{Lexikon} and (\ref{cumulant-one-sided}) we obtain the third moment
\begin{align*}
\mathbb{E}[X^3] &= \kappa_1^3 + 3 \kappa_1 \kappa_2 + \kappa_3
\\ &= \bigg( \Gamma(1-\beta) \frac{\alpha}{\lambda^{1-\beta}} \bigg)^3 + 3 \Gamma(1-\beta) \frac{\alpha}{\lambda^{1-\beta}} \Gamma(2-\beta) \frac{\alpha}{\lambda^{2-\beta}} + \Gamma(3-\beta) \frac{\alpha}{\lambda^{3-\beta}}
\\ &= \Gamma(1-\beta) \frac{\alpha}{\lambda^{3-\beta}} \big[ \Gamma(1-\beta)^2 \alpha^2 \lambda^{2\beta} + 3 (1-\beta) \Gamma(1-\beta) \alpha \lambda^{\beta} + (1-\beta)(2-\beta) \big],
\end{align*}
completing the proof.
\end{proof}

In the sequel, for $\mu \in \mathbb{R}$ and $\sigma^2 > 0$ the function $\Phi_{\mu,\sigma^2}$ denotes the distribution function of the normal distribution $N(\mu,\sigma^2)$. Moreover, $c > 0$ denotes the constant from the Berry-Esseen theorem. The current best estimate is $c \leq 0.4784$, see \cite[Cor.~1]{Korolev}.

\begin{proposition}\label{prop-CLT-with-f}
There exists a function $g : [0,1)^2 \times (0,\infty)^2 \times \mathbb{R} \rightarrow \mathbb{R}$ such that for any fixed $\beta^+,\beta^- \in [0,1)$ we have 
\begin{align}\label{f-conv-to-zero}
g(\beta^+,\beta^-,\lambda^+,\lambda^-,\mu) \rightarrow 0 \quad \text{as } \lambda^+,\lambda^-,\mu \rightarrow 0, 
\end{align}
and for any random variable
\begin{align*}
X \sim {\rm TS}(\alpha^+,\beta^+,\lambda^+;\alpha^-,\beta^-,\lambda^-),
\end{align*}
all $n \in \mathbb{N}$ and any random variable
\begin{align}\label{def-Xn-CLT}
X_n \sim {\rm TS} ( (\sqrt{n}^{2-\beta^+} / \sigma)\alpha^+, \beta^+, \lambda^+ \sigma \sqrt{n}; (\sqrt{n}^{2-\beta^-} / \sigma)\alpha^-, \beta^-, \lambda^- \sigma \sqrt{n} )
\end{align}
we have
\begin{align*}
&\sup_{x \in \mathbb{R}} |G_n(x) - \Phi_{0,1}(x)| 
\\ &\leq 32 c \bigg[ (1-\beta^+)(2-\beta^+) \frac{(1-\beta^-)\sqrt{n}\mu / \sigma^3 + \sqrt{n}\lambda^- / \sigma}{\sqrt{n}\lambda^+ ((1-\beta^-)\sqrt{n}\lambda^+ + (1-\beta^+)\sqrt{n}\lambda^-)} 
\\ &\qquad\quad\, + (1-\beta^-)(2-\beta^-) \frac{(\beta^+ - 1)\sqrt{n}\mu / \sigma^3 + \sqrt{n}\lambda^+ / \sigma}{\sqrt{n}\lambda^-((1-\beta^-)\sqrt{n}\lambda^+ + (1-\beta^+)\sqrt{n}\lambda^-)} 
\\ &\qquad\quad\, + \frac{g(\beta^+,\beta^-,\lambda^+,\lambda^-,\mu)}{\sigma^3} \bigg],
\end{align*}
where $\mu := \mathbb{E}[X]$, $\sigma^2 := {\rm Var}[X]$ and $G_n$ denotes the 
distribution function of the random variable $X_n - \sqrt{n} \mu / \sigma$.
\end{proposition}

\begin{proof}
We define the functions $g_i : [0,1)^2 \times (0,\infty)^2 \times \mathbb{R} \rightarrow \mathbb{R}$, $i=1,2$ as
\begin{align*}
&g_1(\beta^+,\beta^-,\lambda^+,\lambda^-,\mu) := \frac{(1-\beta^-)\mu + \lambda^- \sigma^2}{\lambda^+ ((1-\beta^-)\lambda^+ + (1-\beta^+)\lambda^-)} 
\\ &\quad \times \bigg[ \bigg( \frac{(\lambda^+)^2 ((1-\beta^-)\mu + \lambda^- \sigma^2)}{(1-\beta^-)\lambda^+ + (1-\beta^+)\lambda^-} \bigg)^2 + 3(1-\beta^+) \frac{(\lambda^+)^2 ((1-\beta^-)\mu + \lambda^- \sigma^2)}{(1-\beta^-)\lambda^+ + (1-\beta^+)\lambda^-} \bigg],
\\ &g_2(\beta^+,\beta^-,\lambda^+,\lambda^-,\mu) := \frac{(\beta^+ - 1)\mu + \lambda^+ \sigma^2}{\lambda^-((1-\beta^-)\lambda^+ + (1-\beta^+)\lambda^-)} 
\\ &\quad \times \bigg[ \bigg( \frac{(\lambda^-)^2 ((\beta^+ - 1)\mu + \lambda^+ \sigma^2)}{(1-\beta^-)\lambda^+ + (1-\beta^+)\lambda^-} \bigg)^2 + 3(1-\beta^-) \frac{(\lambda^-)^2 ((\beta^+ - 1)\mu + \lambda^+ \sigma^2)}{(1-\beta^-)\lambda^+ + (1-\beta^+)\lambda^-} \bigg],
\end{align*}
and let $g : [0,1)^2 \times (0,\infty)^2 \times \mathbb{R} \rightarrow \mathbb{R}$ be the function
\begin{align*}
g := 32 ( g_1 + g_2 ).
\end{align*}
Then, for any fixed $\beta^+,\beta^- \in [0,1)$ we have (\ref{f-conv-to-zero}).

Now, let $(Y_j)_{j \in \mathbb{N}}$ be an i.i.d. sequence of random variables with $\mathcal{L}(Y_j) = \mathcal{L}(X)$ for all $j \in \mathbb{N}$. We define the sequence $(S_n)_{n \in \mathbb{N}}$ as $S_n := \sum_{j=1}^n Y_j$ for $n \in \mathbb{N}$. By Lemma~\ref{lemma-TS-consistent} we have
\begin{align*}
\mathcal{L} \bigg( \frac{S_n}{\sigma \sqrt{n}} \bigg) &= {\rm TS} ( (\sqrt{n}^{2-\beta^+} / \sigma^{\beta^+})\alpha^+, \beta^+, \lambda^+ \sigma \sqrt{n}; (\sqrt{n}^{2-\beta^-} / \sigma^{\beta^-})\alpha^-, \beta^-, \lambda^- \sigma \sqrt{n} )
\\ &= \mathcal{L}(X_n) \quad \text{for all $n \in \mathbb{N}$,}
\end{align*}
and therefore
\begin{align*}
\mathcal{L} \bigg( X_n - \frac{\sqrt{n} \mu}{\sigma} \bigg) = \mathcal{L} \bigg( \frac{S_n}{\sigma \sqrt{n}} - \frac{\sqrt{n} \mu}{\sigma} \bigg) = \mathcal{L} \bigg( \frac{S_n - n \mu}{\sigma \sqrt{n}} \bigg) \quad \text{for all $n \in \mathbb{N}$.}
\end{align*}
By the Berry-Esseen theorem (see, e.g., \cite[Thm.~2.4.9]{Durrett}) we have
\begin{align*}
\sup_{x \in \mathbb{R}} |G_n(x) - \Phi_{0,1}(x)| \leq c \frac{\mathbb{E}[|X - \mu|^3]}{\sigma^3 \sqrt{n}} \quad \text{for all $n \in \mathbb{N}$.}
\end{align*}
We have $X = X^+ - X^-$ with independent random variables $X^+ \sim {\rm TS}(\alpha^+,\beta^+,\lambda^+)$ and $X^- \sim {\rm TS}(\alpha^-,\beta^-,\lambda^-)$, and therefore, using H\"{o}lder's inequality and Jensen's inequality we estimate
\begin{align*}
\mathbb{E}[|X - \mu|^3] &= \mathbb{E}[ |X^+ - X^- - (\mathbb{E}[X^+] - \mathbb{E}[X^-])|^3 ] 
\\ &\leq \mathbb{E}[ ( X^+ + \mathbb{E}[X^+] + X^- + \mathbb{E}[X^-] )^3 ]
\\ &\leq 4^2 \mathbb{E} [ (X^+)^3 + \mathbb{E}[X^+]^3 + (X^-)^3 + \mathbb{E}[X^-]^3 ]
\\ &= 16 \big( \mathbb{E}[(X^+)^3] + \mathbb{E}[X^+]^3 + \mathbb{E}[(X^-)^3] + \mathbb{E}[X^-]^3 \big) 
\\ &\leq 32 \big( \mathbb{E} [ (X^+)^3 ] + \mathbb{E} [ (X^-)^3 ] \big).
\end{align*}
Using Lemma~\ref{lemma-comp-single-TS} we have
\begin{align*}
\mathbb{E}[(X^+)^3] &= \Gamma(1-\beta^+) \frac{\alpha^+}{(\lambda^+)^{3-\beta^+}} \big[ \Gamma(1-\beta^+)^2 (\alpha^+)^2 (\lambda^+)^{2 \beta^+} 
\\ &\quad + 3 (1-\beta^+) \Gamma(1-\beta^+) \alpha^+ (\lambda^+)^{\beta^+} + (1-\beta^+)(2-\beta^+) \big],
\\ \mathbb{E}[(X^-)^3] &= \Gamma(1-\beta^-) \frac{\alpha^-}{(\lambda^-)^{3-\beta^-}} \big[ \Gamma(1-\beta^-)^2 (\alpha^-)^2 (\lambda^-)^{2 \beta^-} 
\\ &\quad + 3 (1-\beta^-) \Gamma(1-\beta^-) \alpha^- (\lambda^-)^{\beta^-} + (1-\beta^-)(2-\beta^-) \big].
\end{align*}
Inserting identities (\ref{sol-alpha-1}), (\ref{sol-alpha-2}) from Lemma~\ref{lemma-det-alpha} yields
\begin{align*}
\mathbb{E}[(X^+)^3] &= (1-\beta^+)(2-\beta^+) \frac{(1-\beta^-)\mu + \lambda^- \sigma^2}{\lambda^+ ((1-\beta^-)\lambda^+ + (1-\beta^+)\lambda^-)} 
\\ &\quad + g_1(\beta^+,\beta^-,\lambda^+,\lambda^-,\mu),
\\ \mathbb{E}[(X^-)^3] &= (1-\beta^-)(2-\beta^-) \frac{(\beta^+ - 1)\mu + \lambda^+ \sigma^2}{\lambda^-((1-\beta^-)\lambda^+ + (1-\beta^+)\lambda^-)} 
\\ &\quad + g_2(\beta^+,\beta^-,\lambda^+,\lambda^-,\mu).
\end{align*}
Therefore, for all $n \in \mathbb{N}$ we conclude
\begin{align*}
&\sup_{x \in \mathbb{R}} |G_n(x) - \Phi_{0,1}(x)| \leq \frac{32 c}{\sigma^3 \sqrt{n}} \bigg[ (1-\beta^+)(2-\beta^+) \frac{(1-\beta^-)\mu + \lambda^- \sigma^2}{\lambda^+ ((1-\beta^-)\lambda^+ + (1-\beta^+)\lambda^-)}
\\ &\quad+ (1-\beta^-)(2-\beta^-) \frac{(\beta^+ - 1)\mu + \lambda^+ \sigma^2}{\lambda^-((1-\beta^-)\lambda^+ + (1-\beta^+)\lambda^-)} + g(\beta^+,\beta^-,\lambda^+,\lambda^-,\mu) \bigg]
\\ &= 32 c \bigg[ (1-\beta^+)(2-\beta^+) \frac{(1-\beta^-)\sqrt{n}\mu / \sigma^3 + \sqrt{n}\lambda^- / \sigma}{\sqrt{n}\lambda^+ ((1-\beta^-)\sqrt{n}\lambda^+ + (1-\beta^+)\sqrt{n}\lambda^-)} 
\\ &\qquad\quad\, + (1-\beta^-)(2-\beta^-) \frac{(\beta^+ - 1)\sqrt{n}\mu / \sigma^3 + \sqrt{n}\lambda^+ / \sigma}{\sqrt{n}\lambda^-((1-\beta^-)\sqrt{n}\lambda^+ + (1-\beta^+)\sqrt{n}\lambda^-)} 
\\ &\qquad\quad\, + \frac{g(\beta^+,\beta^-,\lambda^+,\lambda^-,\mu)}{\sigma^3} \bigg],
\end{align*}
finishing the proof.
\end{proof}

\begin{proposition}\label{prop-unit-var}
For any random variable
\begin{align*}
X \sim {\rm TS}(\alpha^+,\beta^+,\lambda^+;\alpha^-,\beta^-,\lambda^-)
\end{align*}
with ${\rm Var}[X] = 1$ we have
\begin{align*}
\sup_{x \in \mathbb{R}} |G_{X - \mu}(x) - \Phi_{0,1}(x)| &\leq 32 c \bigg[ (1-\beta^+)(2-\beta^+) \frac{(1-\beta^-)\mu + \lambda^-}{\lambda^+ ((1-\beta^-)\lambda^+ + (1-\beta^+)\lambda^-)} 
\\ &\qquad + (1-\beta^-)(2-\beta^-) \frac{(\beta^+ - 1)\mu + \lambda^+}{\lambda^-((1-\beta^-)\lambda^+ + (1-\beta^+)\lambda^-)} \bigg],
\end{align*}
where $\mu := \mathbb{E}[X]$.
\end{proposition}

\begin{proof}
Let $g : [0,1)^2 \times (0,\infty)^2 \times \mathbb{R} \rightarrow \mathbb{R}$ the function from Proposition~\ref{prop-CLT-with-f}.
It suffices to show that for each $\epsilon > 0$ we have
\begin{equation}\label{rate-epsilon}
\begin{aligned}
\sup_{x \in \mathbb{R}} |G_{X-\mu}(x) - \Phi_{0,1}(x)| &\leq 32 c \bigg[ (1-\beta^+)(2-\beta^+) \frac{(1-\beta^-)\mu / \sigma^3 + \lambda^- / \sigma}{\lambda^+ ((1-\beta^-)\lambda^+ + (1-\beta^+)\lambda^-)} 
\\ &+ (1-\beta^-)(2-\beta^-) \frac{(\beta^+ - 1)\mu / \sigma^3 + \lambda^+ / \sigma}{\lambda^-((1-\beta^-)\lambda^+ + (1-\beta^+)\lambda^-)} + \epsilon \bigg].
\end{aligned}
\end{equation}
Let $\epsilon > 0$ be arbitrary. There exists $n \in \mathbb{N}$ such that
\begin{align}\label{est-f-epsilon}
g(\beta^+,\beta^-,\lambda^+ / \sqrt{n}, \lambda^- / \sqrt{n}, \mu / \sqrt{n}) \leq \epsilon.
\end{align}
Let $Y$ be a random variable
\begin{align*}
Y \sim {\rm TS}( \alpha^+ / \sqrt{n}^{2-\beta^+}, \beta^+, \lambda^+ / \sqrt{n}; \alpha^- / \sqrt{n}^{2-\beta^-}, \beta^-, \lambda^- / \sqrt{n} ).
\end{align*}
Applying Lemma~\ref{lemma-transform-para} with $\rho = 1/n$ and $\tau = 1$ we obtain 
\begin{align*}
\mathbb{E}[Y] = \mu / \sqrt{n} \quad \text{and} \quad {\rm Var}[Y] = 1 \quad \text{for all $n \in \mathbb{N}$.} 
\end{align*}
Hence, defining the random variable $Y_n$ according to (\ref{def-Xn-CLT}), we have
\begin{align*}
\mathcal{L}(Y_n) &= {\rm TS}( \sqrt{n}^{2-\beta^+} (\alpha^+ / \sqrt{n}^{2-\beta^+}), \beta^+, \sqrt{n} (\lambda^+ / \sqrt{n}); \\ &\qquad \quad \sqrt{n}^{2-\beta^-} (\alpha^- / \sqrt{n}^{2-\beta^-}), \beta^-, \sqrt{n} (\lambda^- / \sqrt{n}) )
\\ &= {\rm TS}(\alpha^+,\beta^+,\lambda^+;\alpha^-,\beta^-,\lambda^-) = \mathcal{L}(X).
\end{align*}
By Proposition~\ref{prop-CLT-with-f} we deduce
\begin{align*}
&\sup_{x \in \mathbb{R}} |G_{X-\mu}(x) - \Phi_{0,1}(x)| = \sup_{x \in \mathbb{R}} |G_{Y_n - \mu}(x) - \Phi_{0,1}(x)| 
\\ &\leq 32 c \bigg[ (1-\beta^+)(2-\beta^+) \frac{(1-\beta^-)\mu + \lambda^-}{\lambda^+ ((1-\beta^-)\lambda^+ + (1-\beta^+)\lambda^-)} 
\\ &\qquad\quad\, + (1-\beta^-)(2-\beta^-) \frac{(\beta^+ - 1)\mu + \lambda^+}{\lambda^-((1-\beta^-)\lambda^+ + (1-\beta^+)\lambda^-)} 
\\ &\qquad\quad\, + g(\beta^+,\beta^-,\lambda^+ / \sqrt{n},\lambda^- / \sqrt{n},\mu / \sqrt{n}) \bigg],
\end{align*}
and, by virtue of estimate (\ref{est-f-epsilon}), we arrive at (\ref{rate-epsilon}).
\end{proof}

\begin{theorem}\label{thm-CLT-TS}
For any random variable
\begin{align*}
X \sim {\rm TS}(\alpha^+,\beta^+,\lambda^+;\alpha^-,\beta^-,\lambda^-)
\end{align*}
we have the estimate
\begin{align*}
\sup_{x \in \mathbb{R}} |G_X(x) - \Phi_{\mu,\sigma^2}(x)| &\leq 32 c \bigg[ (1-\beta^+)(2-\beta^+) \frac{(1-\beta^-)\mu / \sigma^3 + \lambda^- / \sigma}{\lambda^+ ((1-\beta^-)\lambda^+ + (1-\beta^+)\lambda^-)} 
\\ &\qquad + (1-\beta^-)(2-\beta^-) \frac{(\beta^+ - 1)\mu / \sigma^3 + \lambda^+ / \sigma}{\lambda^-((1-\beta^-)\lambda^+ + (1-\beta^+)\lambda^-)} \bigg],
\end{align*}
where $\mu := \mathbb{E}[X]$ and $\sigma^2 := {\rm Var}[X]$.
\end{theorem}

\begin{proof}
By Lemma~\ref{lemma-TS-consistent}, the random variable $X / \sigma$ has the distribution
\begin{align*}
X / \sigma \sim {\rm TS}(\alpha^+ / {\sigma}^{\beta^+},\beta^+,\lambda^+ \sigma;\alpha^- / {\sigma}^{\beta^-},\beta^-,\lambda^- \sigma),
\end{align*}
and applying Lemma~\ref{lemma-transform-para} with $\rho = 1$ and $\tau = \sigma$ yields that
\begin{align*}
\mathbb{E}[X / \sigma] = \mu / \sigma \quad \text{and} \quad {\rm Var}[X / \sigma] = 1. 
\end{align*}
Moreover, since
\begin{align*}
G_X(x) = G_{X / \sigma - \mu / \sigma}\bigg( \frac{x-\mu}{\sigma} \bigg) \quad \text{and} \quad \Phi_{\mu,\sigma^2}(x) = \Phi_{0,1} \bigg( \frac{x-\mu}{\sigma} \bigg) \quad \text{for all $x \in \mathbb{R}$,}
\end{align*}
applying Proposition~\ref{prop-unit-var} gives us
\begin{align*}
&\sup_{x \in \mathbb{R}} |G_X(x) - \Phi_{\mu,\sigma^2}(x)| = \sup_{x \in \mathbb{R}} |G_{X / \sigma - \mu / \sigma}(x) - \Phi_{0,1}(x)|
\\ &\leq 32 c \bigg[ (1-\beta^+)(2-\beta^+) \frac{(1-\beta^-)\mu / \sigma + \lambda^- \sigma}{\lambda^+ \sigma ((1-\beta^-)\lambda^+ \sigma + (1-\beta^+)\lambda^- \sigma)} 
\\ &\qquad\quad\, + (1-\beta^-)(2-\beta^-) \frac{(\beta^+ - 1)\mu / \sigma + \lambda^+ \sigma}{\lambda^- \sigma ((1-\beta^-)\lambda^+ \sigma + (1-\beta^+)\lambda^- \sigma)} \bigg],
\\ &= 32 c \bigg[ (1-\beta^+)(2-\beta^+) \frac{(1-\beta^-)\mu / \sigma^3 + \lambda^- / \sigma}{\lambda^+ ((1-\beta^-)\lambda^+ + (1-\beta^+)\lambda^-)} 
\\ &\qquad\quad\, + (1-\beta^-)(2-\beta^-) \frac{(\beta^+ - 1)\mu / \sigma^3 + \lambda^+ / \sigma}{\lambda^-((1-\beta^-)\lambda^+ + (1-\beta^+)\lambda^-)} \bigg],
\end{align*}
which completes the proof.
\end{proof}

Theorem~\ref{thm-CLT-TS} tells us, how close the distribution of a tempered stable random variable $X$ is to the normal distribution $N(\mu,\sigma^2)$ with $\mu = \mathbb{E}[X]$ and $\sigma^2 = {\rm Var}[X]$. In the upcoming result, for given values of $\mu \in \mathbb{R}$ and $\sigma^2 > 0$ we construct a sequence of tempered stable distributions, which converges weakly to $N(\mu,\sigma^2)$, and provide a convergence rate.

\begin{corollary}\label{cor-mu-sigma-distribution}
Let  $\mu \in \mathbb{R}$, $\sigma^2 > 0$ and
\begin{align*}
(\beta^+,\lambda^+;\beta^-,\lambda^-) \in ([0,1) \times (0,\infty))^2
\end{align*}
be arbitrary. For each $n \in \mathbb{N}$ we define
\begin{align}\label{def-new-para-1}
\alpha_n^+ &:= \frac{(\lambda^+)^{2-\beta^+} ( (1-\beta^-) \mu + \lambda^-  \sigma^2 \sqrt{n})}{\Gamma(1-\beta^+) ((1-\beta^-)\lambda^+ + (1-\beta^+)\lambda^-)} \sqrt{n}^{1-\beta^+},
\\ \label{def-new-para-2} \lambda_n^+ &:= \lambda^+ \sqrt{n},
\\ \label{def-new-para-3} \alpha_n^- &:= \frac{(\lambda^-)^{2-\beta^-} ( (\beta^+ - 1) \mu + \lambda^+ \sigma^2 \sqrt{n})}{\Gamma(1-\beta^-) ((1-\beta^-)\lambda^+ + (1-\beta^+)\lambda^-)} \sqrt{n}^{1-\beta^-},
\\ \label{def-new-para-4} \lambda_n^- &:= \lambda^- \sqrt{n}.
\end{align}
Then, there exists an index $n_0 \in \mathbb{N}$ with $\alpha_n^+ > 0$ and $\alpha_n^- > 0$ for all integers $n \geq n_0$, and for any sequence $(X_n)_{n \geq n_0}$ of random variables with 
\begin{align*}
X_n \sim {\rm TS}(\alpha_n^+,\beta^+,\lambda_n^+;\alpha_n^-,\beta^-,\lambda_n^-), \quad n \geq n_0
\end{align*}
we have the estimate
\begin{equation}\label{conv-rate-distribution}
\begin{aligned}
&\sup_{x \in \mathbb{R}} |G_{X_n}(x) - \Phi_{\mu,\sigma^2}(x)|
\\ &\leq \frac{32 c}{\sqrt{n}} \bigg[ \frac{(1-\beta^+)(2-\beta^+)}{\lambda^+((1-\beta^-)\lambda^+ + (1-\beta^+)\lambda^-)} \bigg( \frac{(1-\beta^-) \mu}{\sigma^3 \sqrt{n}} + \frac{\lambda^-}{\sigma} \bigg)
\\ &\quad\quad\quad\,\, + \frac{(1-\beta^-)(2-\beta^-)}{\lambda^+((1-\beta^-)\lambda^+ + (1-\beta^+)\lambda^-)} \bigg( \frac{(\beta^+ - 1) \mu}{\sigma^3 \sqrt{n}} + \frac{\lambda^+}{\sigma} \bigg) \bigg] \rightarrow 0 \quad \text{for $n \rightarrow \infty$.}
\end{aligned}
\end{equation}
\end{corollary}

\begin{proof}
The existence of an index $n_0 \in \mathbb{N}$ with $\alpha_n^+ > 0$ and $\alpha_n^- > 0$ for all $n \geq n_0$ immediately follows from the Definitions (\ref{def-new-para-1}), (\ref{def-new-para-3}) of $\alpha_n^+, \alpha_n^-$. By Lemma~\ref{lemma-det-alpha} we have $\mathbb{E}[X_n] = \mu$ and ${\rm Var}[X_n] = \sigma^2$ for all $n \geq n_0$. Applying Theorem~\ref{thm-CLT-TS} yields the asserted estimate (\ref{conv-rate-distribution}).
\end{proof}

\begin{remark}
The Definitions (\ref{def-new-para-1})--(\ref{def-new-para-4}) imply that
\begin{align*}
&\Gamma(1-\beta^+) \frac{\alpha_n^+}{(\lambda_n^+)^{1-\beta^+}} - \Gamma(1-\beta^-) \frac{\alpha_n^-}{(\lambda_n^-)^{1-\beta^-}}
\\ &= \frac{\lambda^+ ( (1-\beta^-) \mu + \lambda^-  \sigma^2 \sqrt{n})}{(1-\beta^-)\lambda^+ + (1-\beta^+)\lambda^-} 
- \frac{\lambda^- ( (\beta^+ - 1) \mu + \lambda^+ \sigma^2 \sqrt{n})}{(1-\beta^-)\lambda^+ + (1-\beta^+)\lambda^-} = \mu
\end{align*}
for all $n \geq n_0$, the convergences
\begin{align*}
&\frac{\alpha_n^+}{(\lambda_n^+)^{2-\beta^+}} \rightarrow \frac{\lambda^- \sigma^2}{\Gamma(1-\beta^+)((1-\beta^-)\lambda^+ + (1-\beta^+)\lambda^-)} =: (\sigma^+)^2,
\\ &\frac{\alpha_n^-}{(\lambda_n^-)^{2-\beta^-}} \rightarrow \frac{\lambda^+ \sigma^2}{\Gamma(1-\beta^-)((1-\beta^-)\lambda^+ + (1-\beta^+)\lambda^-)} =: (\sigma^-)^2
\end{align*}
for $n \rightarrow \infty$ as well as
\begin{align*}
\Gamma(2-\beta^+) (\sigma^+)^2 + \Gamma(2-\beta^-) (\sigma^-)^2 = \frac{(1-\beta^+) \lambda^- + (1-\beta^-) \lambda^+}{(1-\beta^-)\lambda^+ + (1-\beta^+)\lambda^-} \sigma^2 = \sigma^2.
\end{align*}
Consequently, conditions (\ref{CLT-cond-1}), (\ref{CLT-cond-2}) are satisfied, and hence, Proposition~\ref{prop-closure} yields the weak convergence (\ref{weak-normal}). In addition, Corollary~\ref{cor-mu-sigma-distribution} provides the convergence rate (\ref{conv-rate-distribution}).
\end{remark}

\begin{theorem}\label{thm-CLT-TS-process}
For any tempered stable process
\begin{align*}
X \sim {\rm TS}(\alpha^+,\beta^+,\lambda^+;\alpha^-,\beta^-,\lambda^-)
\end{align*}
we have the estimate
\begin{equation}\label{conv-rate-process}
\begin{aligned}
&\sup_{x \in \mathbb{R}} |G_{X_t}(x) - G_{W_t}(x)| \leq \frac{32 c}{\sqrt{t}} \bigg[ (1-\beta^+)(2-\beta^+) \frac{(1-\beta^-)\mu / \sigma^3 + \lambda^- / \sigma}{\lambda^+ ((1-\beta^-)\lambda^+ + (1-\beta^+)\lambda^-)} 
\\ &\quad + (1-\beta^-)(2-\beta^-) \frac{(\beta^+ - 1)\mu / \sigma^3 + \lambda^+ / \sigma}{\lambda^-((1-\beta^-)\lambda^+ + (1-\beta^+)\lambda^-)} \bigg] \rightarrow 0 \quad \text{for $t \rightarrow \infty$,}
\end{aligned}
\end{equation}
where $\mu := \mathbb{E}[X_1]$, $\sigma^2 := {\rm Var}[X_1]$ and $W$ is a Brownian motion with $W_1 \sim N(\mu,\sigma^2)$.
\end{theorem}

\begin{proof}
Noting that by (\ref{expectation-TS}), (\ref{variance-TS}) and (\ref{distribution-TS-time}) we have $\mathbb{E}[X_t] = t \mu$ and ${\rm Var}[X_t] = t \sigma^2$ for all $t > 0$, applying Theorem~\ref{thm-CLT-TS} yields
\begin{align*}
&\sup_{x \in \mathbb{R}} |G_{X_t}(x) - G_{W_t}(x)| = \sup_{x \in \mathbb{R}} |G_{X_t}(x) - \Phi_{\mu t,\sigma^2 t}(x)| 
\\ &\leq 32 c \bigg[ (1-\beta^+)(2-\beta^+) \frac{(1-\beta^-)t \mu / t^{3/2} \sigma^3 + \lambda^- / t^{1/2} \sigma}{\lambda^+ ((1-\beta^-)\lambda^+ + (1-\beta^+)\lambda^-)} 
\\ &\qquad\quad\, + (1-\beta^-)(2-\beta^-) \frac{(\beta^+ - 1)t\mu / t^{3/2} \sigma^3 + \lambda^+ / t^{1/2} \sigma}{\lambda^-((1-\beta^-)\lambda^+ + (1-\beta^+)\lambda^-)} \bigg]
\\ &= \frac{32 c}{\sqrt{t}} \bigg[ (1-\beta^+)(2-\beta^+) \frac{(1-\beta^-)\mu / \sigma^3 + \lambda^- / \sigma}{\lambda^+ ((1-\beta^-)\lambda^+ + (1-\beta^+)\lambda^-)} 
\\ &\qquad\quad\, + (1-\beta^-)(2-\beta^-) \frac{(\beta^+ - 1)\mu / \sigma^3 + \lambda^+ / \sigma}{\lambda^-((1-\beta^-)\lambda^+ + (1-\beta^+)\lambda^-)} \bigg] \rightarrow 0 \quad \text{for $t \rightarrow \infty$,}
\end{align*}
completing the proof.
\end{proof}

\begin{corollary}
If we choose $\mu \in \mathbb{R}$, $\sigma^2 > 0$ and
\begin{align*}
(\beta^+,\lambda^+;\beta^-,\lambda^-) \in ([0,1) \times (0,\infty))^2
\end{align*}
and choose
\begin{align}\label{cond-alpha-plus}
\alpha^+ &:= \frac{(\lambda^+)^{2-\beta^+} ((1-\beta^-)\mu + \lambda^- \sigma^2)}{\Gamma(1-\beta^+)((1-\beta^-)\lambda^+ + (1-\beta^+)\lambda^-)} > 0,
\\ \label{cond-alpha-minus} \alpha^- &:= \frac{(\lambda^-)^{2-\beta^-} ((\beta^+ - 1)\mu + \lambda^+ \sigma^2)}{\Gamma(1-\beta^-)((1-\beta^-)\lambda^+ + (1-\beta^+)\lambda^-)} > 0,
\end{align}
then for any tempered stable process 
\begin{align*}
X \sim {\rm TS}(\alpha^+,\beta^+,\lambda^+;\alpha^-,\beta^-,\lambda^-)
\end{align*}
estimate (\ref{conv-rate-process}) is valid.
\end{corollary}

\begin{proof}
By Lemma~\ref{lemma-det-alpha} we have $\mathbb{E}[X_1] = \mu$ and ${\rm Var}[X_1] = \sigma^2$. Applying Theorem~\ref{thm-CLT-TS-process} yields the desired estimate (\ref{conv-rate-process}).
\end{proof}

\begin{remark}
Note that conditions (\ref{cond-alpha-plus}), (\ref{cond-alpha-minus}) are always satisfied for $\mu = 0$.
\end{remark}

\section{Laws of large numbers for tempered stable distributions}\label{sec-lln}

In this section, we present laws of large numbers for tempered stable distributions. These results will be useful in order to determine parameters from the observation of a typical trajectory of a tempered stable process.

\begin{proposition}\label{prop-lln}
We have the weak convergence
\begin{align}\label{lln-weak}
{\rm TS}(n^{1-\beta^+}\alpha^+,\beta^+,n \lambda^+;n^{1-\beta^-}\alpha^-,\beta^-,n \lambda^-) \overset{w}{\rightarrow} \delta_{\mu} \quad \text{for $n \rightarrow \infty$,}
\end{align}
where the number $\mu \in \mathbb{R}$ equals the mean
\begin{align*}
\mu = \Gamma(1 - \beta^+) \frac{\alpha^+}{(\lambda^+)^{1 - \beta^+}} - \Gamma(1 - \beta^-) \frac{\alpha^-}{(\lambda^-)^{1 - \beta^-}}.
\end{align*}
\end{proposition}

\begin{proof}
Let $(X_j)_{j \in \mathbb{N}}$ be an i.i.d. sequence with
\begin{align*}
X_j \sim {\rm TS}(\alpha^+,\beta^+,\lambda^+;\alpha^-,\beta^-,\lambda^-) \quad \text{for $j \in \mathbb{N}$.}
\end{align*}
By (\ref{expectation-TS}) we have $\mathbb{E}[X_j] = \mu$ for all $j \in \mathbb{N}$, and Lemma~\ref{lemma-TS-consistent} yields that
\begin{align*}
\frac{1}{n} \sum_{j=1}^n X_j \sim {\rm TS}(n^{1-\beta^+}\alpha^+,\beta^+,n \lambda^+;n^{1-\beta^-}\alpha^-,\beta^-,n \lambda^-), \quad n \in \mathbb{N}.
\end{align*}
Using the law of large numbers, we deduce the weak convergence (\ref{lln-weak}).
\end{proof}

\begin{remark}
Note that we can alternatively establish the proof of Proposition~\ref{prop-lln} by applying the weak convergence (\ref{weak-LLN}) from Proposition~\ref{prop-closure}.
\end{remark}

In the sequel, we shall establish results in order to determine the parameters from the observation of one typical sample path of a tempered stable process
\begin{align*}
X \sim {\rm TS}(\alpha^+,\beta^+,\lambda^+;\alpha^-,\beta^-,\lambda^-).
\end{align*}
For this, it suffices to treat the case of one-sided tempered stable processes. Indeed, we can decompose $X = X^+ - X^-$, where 
$X^+ \sim {\rm TS}(\alpha^+,\beta^+,\lambda^+)$ and $X^- \sim {\rm TS}(\alpha^-,\beta^-,\lambda^-)$
are two independent one-sided tempered stable subordinators. Since $X_t = \sum_{s \leq t} \Delta X_s$, the observation of a trajectory of $X$ also provides the respective trajectories of $X^+,X^-$, which are given by
\begin{align*}
X_t^+ = \sum_{s \leq t} (\Delta X_s)^+ \quad \text{and} \quad X_t^- = \sum_{s \leq t} (\Delta X_s)^-.
\end{align*}

\begin{remark}
Note that our standing assumption $\beta^+,\beta^- \in [0,1)$ is crucial, because in the infinite variation case an observation of $X$ does not provide the trajectories of the components $X^+$ and $X^-$. 
\end{remark}

Now, let $X \sim {\rm TS}(\alpha,\beta,\lambda)$ be a tempered stable process with $\beta \in [0,1)$. Suppose we observe the process at discrete time points, say $X_{\Delta k}$, $k \in \mathbb{N}$ for some constant $\Delta > 0$. By (\ref{X-Delta-bullet}) we may assume, without loss of generality, that $\Delta = 1$. Setting $m_j := \mathbb{E}[X_1^j]$ for $j=1,2,3$, by the law of large numbers for $n \rightarrow \infty$ we have almost surely
\begin{align*}
\frac{1}{n} \sum_{k=1}^n (X_k - X_{k-1})^j \rightarrow m_j, \quad j=1,2,3.
\end{align*}
Hence, we obtain the moments $m_1,m_2,m_3$ by inspecting a typical sample path of $X$. According to \cite[p.~346]{Lexikon}, the cumulants (\ref{cumulant-one-sided}) are given by
\begin{align*}
\kappa_1 &= m_1,
\\ \kappa_2 &= m_2 - m_1^2,
\\ \kappa_3 &= m_3 - 3 m_1 m_2 + 2 m_1^3.
\end{align*}
By means of the cumulants, we can determine the parameters $\alpha,\beta,\lambda$, as the following result shows:

\begin{proposition}\label{prop-estimate-all}
The parameters $\alpha,\beta,\lambda$ are given by
\begin{align}\label{beta-formula}
\beta &= 1 - \frac{\kappa_2^2}{\kappa_1 \kappa_3 - \kappa_2^2},
\\ \label{lambda-formula} \lambda &= (1-\beta) \frac{\kappa_1}{\kappa_2},
\\ \label{alpha-formula} \alpha &= \frac{\lambda^{1-\beta}}{\Gamma(1-\beta)} \kappa_1.
\end{align}
\end{proposition}

\begin{proof}
According to (\ref{cumulant-one-sided}) we have
\begin{align}\label{system-1}
\Gamma(1-\beta) \alpha &= \kappa_1 \lambda^{1-\beta},
\\ \label{system-2} (1-\beta) \Gamma(1-\beta) \alpha &= \kappa_2 \lambda^{2-\beta},
\\ \label{system-3} (2-\beta)(1-\beta) \Gamma(1-\beta) \alpha &= \kappa_3 \lambda^{3-\beta}.
\end{align}
Equation (\ref{system-1}) yields (\ref{alpha-formula}), and inserting (\ref{alpha-formula}) into (\ref{system-2}) gives us (\ref{lambda-formula}). Note that
\begin{align}\label{kappa-ineqn}
\kappa_1 \kappa_3 > \kappa_2^2.
\end{align}
Indeed, by (\ref{cumulant-one-sided}) we have
\begin{align*}
\kappa_1 \kappa_3 &= \Gamma(1-\beta) \frac{\alpha}{\lambda^{1-\beta}} \Gamma(3-\beta) \frac{\alpha}{\lambda^{3-\beta}} = (2-\beta) \bigg( \Gamma(1-\beta) \Gamma(2-\beta) \frac{\alpha}{\lambda^{2-\beta}} \bigg)^2 
\\ &= (2-\beta) \kappa_2^2 > \kappa_2^2,
\end{align*}
because $\beta \in [0,1)$. Inserting (\ref{alpha-formula}) into (\ref{system-3}) we arrive at (\ref{beta-formula}).
\end{proof}

\begin{remark}
Note that (\ref{kappa-ineqn}) ensures that the right-hand side of (\ref{beta-formula}) is well-defined. We emphasize that for $\beta \in [1,2)$ the parameters $\alpha,\beta,\gamma$ cannot be computed by means of the first three cumulants $\kappa_1,\kappa_2,\kappa_3$, because $\kappa_1$ does not depend on the parameters, see Remark~\ref{remark-cumulant-inf-var}. For analogous identities in this case, we would rather consider the cumulants $\kappa_2,\kappa_3,\kappa_4$.
\end{remark}

For the next result, we suppose that $\beta \in (0,1)$. Let $\varphi_{\beta} : \mathbb{R}_+ \rightarrow (0,\infty)$ be the strictly decreasing function
\begin{align*}
\varphi_{\beta}(x) = (1 + \beta x)^{-\frac{1}{\beta}}, \quad x \in \mathbb{R}_+.
\end{align*}
Let $N$ be the random measure associated to the jumps of $X$. Then $N$ is a homogeneous Poisson random measure with compensator $dt \otimes F(dx)$. We define the random variables $(Y_n)_{n \in \mathbb{N}}$ as
\begin{align*}
Y_n := N((0,1] \times (\varphi_{\beta}(n+1),\varphi_{\beta}(n)]), \quad n \in \mathbb{N}.
\end{align*}

\begin{proposition}\label{prop-estimate-alpha}
We have almost surely the convergence
\begin{align}\label{as-convergence}
\frac{1}{n} \sum_{j=1}^n Y_j \rightarrow \alpha.
\end{align}
\end{proposition}

\begin{proof}
The random variables $(Y_n)_{n \in \mathbb{N}}$ are independent and have a Poisson distribution with mean $\alpha_n = F((\varphi_{\beta}(n+1),\varphi_{\beta}(n)])$ for $n \in \mathbb{N}$. The function $\varphi_{\beta}$ has the derivative
\begin{align*}
\varphi_{\beta}'(x) = -\frac{1}{\beta} (1 + \beta x)^{- \frac{1}{\beta} - 1} \beta = -(1 + \beta x)^{-\frac{1 + \beta}{\beta}}, \quad x \in \mathbb{R}_+
\end{align*}
and thus, by substitution and Lebesgue's dominated convergence theorem we obtain
\begin{align*}
\alpha_n &= F((\varphi_{\beta}(n+1),\varphi_{\beta}(n)]) = \alpha \int_{\varphi_{\beta}(n+1)}^{\varphi_{\beta}(n)} \frac{e^{-\lambda x}}{x^{1 + \beta}} dx = \alpha \int_{n+1}^n \frac{e^{-\lambda \varphi_{\beta}(y)}}{\varphi_{\beta}(y)^{1 + \beta}} \varphi_{\beta}'(y) dy
\\ &= \alpha \int_{n}^{n+1} \frac{e^{-\lambda \varphi_{\beta}(y)}}{(1 + \beta y)^{- \frac{1 + \beta}{\beta}}} (1 + \beta y)^{-\frac{1 + \beta}{\beta}} dy = \alpha \int_{n}^{n+1} e^{-\lambda \varphi_{\beta}(y)} dy 
\\ &= \alpha \int_0^1 e^{-\lambda \varphi_{\beta}(n + y)} dy \rightarrow \alpha \quad \text{for $n \rightarrow \infty$.}
\end{align*}
Consequently, we have
\begin{align*}
\sum_{n=1}^{\infty} \frac{{\rm Var}[Y_n]}{n^2} = \sum_{n=1}^{\infty} \frac{\alpha_n}{n^2} < \infty. 
\end{align*}
Set $S_n = Y_1 + \ldots + Y_n$ for $n \in \mathbb{N}$. Applying Kolmogorov's strong law of large numbers, see \cite[Thm.~IV.3.2]{Shiryaev}, yields almost surely
\begin{align*}
\frac{S_n - \mathbb{E}[S_n]}{n} \rightarrow 0.
\end{align*}
Since we have
\begin{align*}
\frac{\mathbb{E}[S_n]}{n} = \frac{\alpha_1 + \ldots + \alpha_n}{n} \rightarrow \alpha,
\end{align*}
we arrive at the almost sure convergence (\ref{as-convergence}).
\end{proof}

\begin{remark}\label{remark-estimate-alpha}
Let us consider the situation $\beta = 0$. For all $x \in \mathbb{R}_+$ we have
\begin{align*}
\lim_{\beta \rightarrow 0} \varphi_{\beta}(x) = \lim_{\beta \rightarrow 0} (1 + \beta x)^{- \frac{1}{\beta}} = e^{-x}.
\end{align*}
This suggests to take the sequence
\begin{align*}
Y_n := N((0,1] \times (e^{-(n+1)},e^{-n}]), \quad n \in \mathbb{N}
\end{align*}
for a Gamma process $X \sim \Gamma(\alpha,\lambda)$, and then, an analogous result is indeed true, see \cite[Thm.~7.1]{Kuechler-Tappe}.
\end{remark}

\begin{remark}
Proposition~\ref{prop-estimate-alpha} and Remark~\ref{remark-estimate-alpha} show how we can determine the parameters $\alpha, \lambda > 0$ by inspecting a typical sample path of $X$, provided that $\beta \in [0,1)$ is known. First, we determine $\alpha > 0$ according to Proposition~\ref{prop-estimate-alpha} or Remark~\ref{remark-estimate-alpha}. By the strong law of large numbers we have almost surely $X_n / n \rightarrow \mu$, where $\mu > 0$ denotes the expectation
\begin{align*}
\mu = \Gamma(1 - \beta) \frac{\alpha}{\lambda^{1 - \beta}}.
\end{align*}
Now, we obtain the parameter $\lambda > 0$ as
\begin{align*}
\lambda = \bigg( \frac{\alpha \Gamma(1 - \beta)}{\mu} \bigg)^{\frac{1}{1 - \beta}}.
\end{align*}
\end{remark}

\section{Statistics of tempered stable
distributions}\label{sec-statistics-distributions}

This section is devoted to parameter estimation for tempered stable distributions. Let $X_1,\ldots,X_n$ be an i.i.d. sequence with
\begin{align*}
X_k \sim {\rm TS}(\alpha^+,\beta^+,\lambda^+;\alpha^-,\beta^-,\lambda^-) \quad \text{for $k = 1,\ldots,n$.}
\end{align*}
Suppose we observe a realization $x_1,\ldots,x_n$. We would like to estimate the vector of parameters 
\begin{align*}
\vartheta = (\vartheta_1,\ldots,\vartheta_6) = (\alpha^+,\beta^+,\lambda^+,\alpha^-,\beta^-,\lambda^-) \in D, 
\end{align*}
where the parameter domain $D$ is the open set
\begin{align*}
D = ((0,\infty) \times (0,1) \times (0,\infty))^2.
\end{align*}
For bilateral Gamma distributions (i.e. $\beta^+ = \beta^- = 0$) parameter estimation was provided in \cite{Kuechler-Tappe}.
We perform the method of moments and estimate the $k$-th moments $m_j = \mathbb{E} [X_1^j]$ as
\begin{align*}
\hat{m}_j = \frac{1}{n} \sum_{k=1}^n x_k^j, \quad j=1,\ldots,6.
\end{align*}
According to (\ref{cumulant-kappa-gamma}) the vector 
\begin{align*}
\kappa = (\kappa_1,\ldots,\kappa_6) \in (\mathbb{R} \times (0,\infty))^3
\end{align*}
of cumulants is given by
\begin{align}\label{cumulants-stat}
\kappa_j = \Gamma(j - \beta^+) \frac{\alpha^+}{(\lambda^+)^{j - \beta^+}} + (-1)^j \Gamma(j - \beta^-) \frac{\alpha^-}{(\lambda^-)^{j - \beta^-}}, \quad j = 1,\ldots,6.
\end{align}
Using \cite[p.~346]{Lexikon} we estimate $\kappa$ as the vector
\begin{align*}
\hat{\kappa} = (\hat{\kappa}_1,\ldots,\hat{\kappa}_6) \in (\mathbb{R} \times (0,\infty))^3
\end{align*}
with components
\begin{align*}
\hat{\kappa}_1 &= \hat{m}_1,
\\ \hat{\kappa}_2 &= \hat{m}_2 - \hat{m}_1^2,
\\ \hat{\kappa}_3 &= \hat{m}_3 - 3 \hat{m}_1 \hat{m}_2 + 2 \hat{m}_1^3,
\\ \hat{\kappa}_4 &= \hat{m}_4 - 4 \hat{m}_1 \hat{m}_3 - 3 \hat{m}_2^2 + 12 \hat{m}_1^2 \hat{m}_2 - 6 \hat{m}_1^4,
\\ \hat{\kappa}_5 &= \hat{m}_5 - 5 \hat{m}_1 \hat{m}_4 - 10 \hat{m}_2 \hat{m}_3 + 20 \hat{m}_1^2 \hat{m}_3 + 30 \hat{m}_1 \hat{m}_2^2 - 60 \hat{m}_1^3 \hat{m}_2 + 24 \hat{m}_1^5,
\\ \hat{\kappa}_6 &= \hat{m}_6 - 6 \hat{m}_1 \hat{m}_5 - 15 \hat{m}_2 \hat{m}_4 + 30 \hat{m}_1^2 \hat{m}_4 - 10 \hat{m}_3^2 + 120 \hat{m}_1 \hat{m}_2 \hat{m}_3 
\\ &\quad - 120 \hat{m}_1^3 \hat{m}_3 + 30 \hat{m}_2^3 - 270 \hat{m}_1^2 \hat{m}_2^2 + 360 \hat{m}_1^4 \hat{m}_2 - 120 \hat{m}_1^6.
\end{align*}
and the function $G : (\mathbb{R} \times (0,\infty))^3 \times D \rightarrow \mathbb{R}^6$ as
\begin{align*}
G_j(c,\hat{\vartheta}) &:= \Gamma(j-\hat{\beta}^+) \hat{\alpha}^+ (\hat{\lambda}^-)^{j-\hat{\beta}^-} + (-1)^j \Gamma(j - \hat{\beta}^-) \hat{\alpha}^- (\hat{\lambda}^+)^{j - \hat{\beta}^+}
\\ &\quad - c_j (\hat{\lambda}^+)^{j - \hat{\beta}^+} (\hat{\lambda}^-)^{j - \hat{\beta}^-}, \quad j = 1,\ldots,6,
\end{align*}
where
\begin{align*}
c = (c_1,\ldots,c_6) \quad \text{and} \quad \hat{\vartheta} = (\hat{\alpha}^+,\hat{\beta}^+,\hat{\lambda}^+,\hat{\alpha}^-,\hat{\beta}^-,\hat{\lambda}^-).
\end{align*}
In order to obtain an estimate $\hat{\vartheta}$ for $\vartheta$ we solve the equation
\begin{align}\label{equation}
G(\hat{\kappa},\hat{\vartheta}) = 0, \quad \hat{\vartheta} \in D.
\end{align}
Let us have a closer look at equation (\ref{equation}) concerning existence and uniqueness of solutions. For the following calculations, we have used the Computer Algebra System ``Maxima''.

\begin{lemma}\label{lemma-det}
We have $G \in C^1(D;\mathbb{R}^6)$ and for all $\vartheta \in D$ and $\kappa = \kappa(\vartheta)$ given by (\ref{cumulants-stat}) we have $G(\kappa,\vartheta) = 0$ and $\det \frac{\partial G}{\partial \vartheta}(\kappa,\vartheta) > 0$.
\end{lemma}

\begin{proof}
The definition of $G$ shows that $G \in C^1(D;\mathbb{R}^6)$.
Let $\vartheta \in D$ be arbitrary and let $\kappa = \kappa(\vartheta)$ be given by (\ref{cumulants-stat}). The identity (\ref{cumulants-stat}) yields that $G(\kappa,\vartheta) = 0$. Computing the partial derivatives of $G$ and inserting the vector $(\kappa,\vartheta)$, for $j=1,\ldots,6$ we obtain
\begin{align*}
\frac{\partial G_j}{\partial \alpha^+}(\kappa,\vartheta) &= g_1 (\lambda^-)^{j-1} \prod_{k=1}^{j-1}(k-\beta^+),
\\ \frac{\partial G_j}{\partial \alpha^-}(\kappa,\vartheta) &= g_2 (-1)^j (\lambda^+)^{j-1} \prod_{k=1}^{j-1} (k - \beta^-),
\\ \frac{\partial G_j}{\partial \beta^+}(\kappa,\vartheta) &= g_3 (\lambda^-)^{j-1} \prod_{k=1}^{j-1} (k-\beta^+) 
\bigg( \ln \lambda^+ - \psi(1-\beta^+) - \sum_{k=0}^{j-2} \frac{1}{1 - \beta^+ + k} \bigg),
\\ \frac{\partial G_j}{\partial \beta^-}(\kappa,\vartheta) &= g_4 (-1)^j (\lambda^+)^{j-1} \prod_{k=1}^{j-1} (k-\beta^-) 
\bigg( \ln \lambda^- - \psi(1-\beta^-) - \sum_{k=0}^{j-2} \frac{1}{1 - \beta^- + k} \bigg),
\\ \frac{\partial G_j}{\partial \lambda^+}(\kappa,\vartheta) &= g_5 (\lambda^-)^{j-1} \prod_{k=2}^j (k-\beta^+),
\\ \frac{\partial G_j}{\partial \lambda^-}(\kappa,\vartheta) &= g_6 (-1)^j (\lambda^+)^{j-1} \prod_{k=2}^j (k-\beta^-),
\end{align*}
where $\psi : (0,\infty) \rightarrow \mathbb{R}$ denotes the Digamma function
\begin{align*}
\psi(x) = \frac{d}{dx} \ln \Gamma(x) = \frac{\Gamma'(x)}{\Gamma(x)}, \quad x \in (0,\infty)
\end{align*}
and where the $g_i$, $i = 1,\ldots,6$ are given by
\begin{align*}
g_1 &= (\lambda^-)^{1-\beta^-} \Gamma(1-\beta^+),
\\ g_2 &= (\lambda^+)^{1-\beta^+} \Gamma(1-\beta^-),
\\ g_3 &= \alpha^+ (\lambda^-)^{1-\beta^-} \Gamma(1-\beta^+),
\\ g_4 &= \alpha^- (\lambda^+)^{1-\beta^+} \Gamma(1-\beta^-),
\\ g_5 &= - \alpha^+ (\lambda^+)^{-1} (\lambda^-)^{1-\beta^-} \Gamma(2-\beta^+),
\\ g_6 &= - \alpha^- (\lambda^-)^{-1} (\lambda^+)^{1-\beta^+} \Gamma(2-\beta^-).
\end{align*}
Let $p_1$ be the polynomial
\begin{align*}
p_1(\beta^+,\beta^-) &= - 18 (\beta^+)^3 (\beta^-)^2 + 84 (\beta^+)^3 \beta^- + 168 (\beta^+)^2 (\beta^-)^2 
\\ &\quad - 99 (\beta^+)^3 - 766 (\beta^+)^2 \beta^- - 507 \beta^+ (\beta^-)^2 
\\ &\quad + 885 (\beta^+)^2 + 2243 \beta^+ \beta^- + 504 (\beta^-)^2 - 2522 \beta^+ - 2160 \beta^- + 2356,
\end{align*}
let $p_2$ be the polynomial
\begin{align*}
p_2(\beta^+,\beta^-) &= - 14 (\beta^+)^4 (\beta^-)^3 + 98 (\beta^+)^4 (\beta^-)^2 + 126 (\beta^+)^3 (\beta^-)^3 
\\ &\quad - 227 (\beta^+)^4 \beta^- - 852 (\beta^+)^3 (\beta^-)^2 - 391 (\beta^+)^2 (\beta^-)^3 
\\ &\quad + 175 (\beta^+)^4 + 1908 (\beta^+)^3 \beta^- + 2542 (\beta^+)^2 (\beta^-)^2 + 499 \beta^+ (\beta^-)^3 
\\ &\quad - 1422 (\beta^+)^3 - 5508 (\beta^+)^2 \beta^- - 3076 \beta^+ (\beta^-)^2 - 232 (\beta^-)^3 
\\ &\quad + 3989 (\beta^+)^2 + 6395 \beta^+ \beta^- + 1336 (\beta^-)^2 
- 4490 \beta^+ - 2628 \beta^- + 1772
\end{align*}
and let $p_3$ be the polynomial
\begin{align*}
p_3(\beta^+,\beta^-) &= - 63 (\beta^+)^4 (\beta^-)^3 + 462 (\beta^+)^4 (\beta^-)^2 + 609 (\beta^+)^3 (\beta^-)^3 
\\ &\quad - 1092 (\beta^+)^4 \beta^- - 4286 (\beta^+)^3 (\beta^-)^2 - 2112 (\beta^+)^2 (\beta^-)^3 
\\ &\quad + 837 (\beta^+)^4 + 9686 (\beta^+)^3 \beta^- + 14198 (\beta^+)^2 (\beta^-)^2 + 3111 \beta^+ (\beta^-)^3 
\\ &\quad - 7185 (\beta^+)^3 - 30373 (\beta^+)^2 \beta^- - 19850 \beta^+ (\beta^-)^2 - 1689 (\beta^-)^3 
\\ &\quad + 21587 (\beta^+)^2 + 39598 \beta^+ \beta^- + 10244 (\beta^-)^2 
\\ &\quad - 26507 \beta^+ - 18923 \beta^- + 11748.
\end{align*}
A plot with the Computer Algebra System ``Maxima'' shows that $p_i > 0$ on $(0,1)^2$ for $i=1,2,3$, and so we obtain
\begin{align*}
&\det \frac{\partial G}{\partial \vartheta}(\kappa,\vartheta) = g_1 \cdots g_6 (\lambda^+ \lambda^-)^3 \Big[ (1-\beta^+)^2 (2-\beta^+)^3 (3-\beta^+)^3 (4-\beta^+) (\lambda^-)^9
\\ &\quad + (1-\beta^+)^2 (2-\beta^+)^3 (3-\beta^+) (4-\beta^+) (7 - 3\beta^-) (11 - 3 \beta^+) \lambda^+ (\lambda^-)^8
\\ &\quad + 2 (1 - \beta^+)^2 (2 - \beta^+) (3 - \beta^+) p_1(\beta^+,\beta^-) (\lambda^+)^2 (\lambda^-)^7
\\ &\quad + 6 (2 - \beta^+) (3 - \beta^+) p_2(\beta^+,\beta^-) (\lambda^+)^3 (\lambda^-)^6
\\ &\quad + 2 (2 - \beta^+) (2 - \beta^-) p_3(\beta^+,\beta^-) (\lambda^+)^4 (\lambda^-)^5
\\ &\quad + 2 (2 - \beta^-) (2 - \beta^+) p_3(\beta^-,\beta^+) (\lambda^+)^5 (\lambda^-)^4
\\ &\quad + 6 (2 - \beta^-) (3 - \beta^-) p_2(\beta^-,\beta^+) (\lambda^+)^6 (\lambda^-)^3
\\ &\quad + 2 (1 - \beta^-)^2 (2 - \beta^-) (3 - \beta^-) p_1(\beta^-,\beta^+) (\lambda^+)^7 (\lambda^-)^2
\\ &\quad + (1 - \beta^-)^2 (2 - \beta^-)^3 (3 - \beta^-) (4 - \beta^-) (7 - 3 \beta^+) (11 - 3 \beta^-) (\lambda^+)^8 \lambda^-
\\ &\quad + (1- \beta^-)^2 (2 - \beta^-)^3 (3 - \beta^-)^3 (4 - \beta^-) (\lambda^+)^9 \Big] > 0,
\end{align*}
finishing the proof.
\end{proof}

Taking into account Lemma~\ref{lemma-det}, by the implicit function theorem (see, e.g., \cite[Thm.~8.1]{Zorich}) there exist an open neighborhood $U_{\kappa} \subset (\mathbb{R} \times (0,\infty))^3$ of $\kappa$, an open neighborhood $U_{\vartheta} \subset D$ of $\vartheta$ and a function $g \in C^1(U_{\kappa};U_{\vartheta})$ such for all $(\hat{\kappa},\hat{\vartheta}) \in U_{\kappa} \times U_{\vartheta}$
we have
\begin{align*}
G(\hat{\kappa},g(\hat{\kappa})) = 0 \quad \text{if and only if} \quad \hat{\vartheta} = g(\hat{\kappa}).
\end{align*}
Recall that $n \in \mathbb{N}$ denotes the number of observations of the realization. If $n$ is large enough, then,
by the law of large numbers, we have $\hat{\kappa} \in U_{\kappa}$, and hence $\hat{\vartheta} := g(\hat{\kappa})$ is the
unique $U_{\vartheta}$-valued solution for (\ref{equation}). This is our estimate for the vector $\vartheta$ of parameters.

\section{Analysis of the densities of tempered stable distributions}\label{sec-densities}

In this section, we shall derive structural properties of the densities of tempered stable distributions. This concerns unimodality, smoothness of the densities and their asymptotic behaviour.

First, we deal with one-sided tempered stable distributions. Let
$\eta = {\rm TS}(\alpha,\lambda,\beta)$
be a one-sided tempered stable distribution with $\beta \in (0,1)$. Note that for $\beta = 0$ we would have the well-known Gamma distribution.

In view of the L\'{e}vy measure (\ref{Levy-measure-one-sided}), we can express the characteristic function of $\eta$ as
\begin{align}\label{cf-self-d}
\varphi(z) = \exp \left( \int_{\mathbb{R}} \left( e^{izx} - 1
\right) \frac{k(x)}{x} dx \right), \quad z \in \mathbb{R}
\end{align}
where $k : \mathbb{R} \rightarrow \mathbb{R}$ denotes the function
\begin{align}\label{k-one-sided}
k(x) = \alpha \frac{e^{-\lambda x}}{x^{\beta}}
\mathbbm{1}_{(0,\infty)}(x), \quad x \in \mathbb{R}.
\end{align}
Note that $k > 0$ on $(0,\infty)$ and that $k$ is strictly decreasing on $(0,\infty)$. Furthermore, we have $k(0+) = \infty$ and $\int_{0}^1 |k(x)| dx < \infty$.

It is an immediate consequence of \cite[Cor.~15.11]{Sato} that $\eta$
is selfdecomposable, and hence absolutely continuous according to \cite[Example~27.8]{Sato}. In what follows, we denote by $g$ the density of $\eta$.

Note that $\eta$ is also of class $L$ in the sense of \cite{Sato-Yamazato}. Indeed, identity (\ref{cf-self-d}) shows that the characteristic function is of the form (1.5) in \cite{Sato-Yamazato} with $\gamma_0 = 0$ and $\sigma^2 = 0$.

\begin{theorem}\label{thm-unimodal-one-sided}
The density $g$ of the tempered stable distribution $\eta$ with $\beta \in (0,1)$ is of class $C^{\infty}(\mathbb{R})$ and there exists a point $x_0 \in (0,\infty)$ such that
\begin{align}\label{mode-x0-1}
g'(x) &> 0, \quad x \in (-\infty,x_0)
\\ \label{mode-x0-2} g'(x_0) &= 0,
\\ \label{mode-x0-3} g'(x) &< 0, \quad x \in (x_0,\infty).
\end{align}
Moreover, we have
\begin{equation}\label{location-one-sided}
\begin{aligned}
&\max \bigg\{ \Gamma(1-\beta) \frac{\alpha}{\lambda^{1-\beta}} - \bigg( 3 \Gamma(2-\beta) \frac{\alpha}{\lambda^{2-\beta}} \bigg)^{1/2}, \xi_0 \bigg\} < x_0 
\\ &< \min \bigg\{ \Gamma(1-\beta) \frac{\alpha}{\lambda^{1-\beta}}, \bigg( \frac{\alpha}{1 - \beta} \bigg)^{1 / \beta} \bigg\},
\end{aligned}
\end{equation}
where $\xi_0 \in (0,\infty)$ denotes the unique solution of the fixed point equation
\begin{align}\label{fixed-point-mode}
\alpha^{1/\beta} \exp \bigg( -\frac{\lambda}{\beta} \xi_0 \bigg) = \xi_0.
\end{align}
\end{theorem}

\begin{proof}
Since $k(0+) + |k(0-)| = \infty$, it follows from \cite[Thm.~1.2]{Sato-Yamazato} that $g \in C^{\infty}(\mathbb{R})$. 

Furthermore, since $\int_0^1 |k(x)| dx < \infty$, the distribution $\eta$ is of type ${\rm I}_6$ in the sense of \cite[page~275]{Sato-Yamazato}. According to \cite[Thm.~1.3.vii]{Sato-Yamazato} there exists a point $x_0 \in (0,\infty)$ such that (\ref{mode-x0-1})--(\ref{mode-x0-3}) are satisfied.

Let $X$ be a random variable with $\mathcal{L}(X) = \eta$. By \cite[Thm.~6.1.i]{Sato-Yamazato} and (\ref{E-TS-one-sided}) we have
\begin{align*}
x_0 < \mathbb{E}[X] = \Gamma(1-\beta) \frac{\alpha}{\lambda^{1-\beta}}.
\end{align*}
Furthermore, by using \cite[Thm.~6.1.ii]{Sato-Yamazato} we have
\begin{align*}
\frac{\alpha}{x_0^{\beta} (1-\beta)} = \frac{\alpha}{x_0} \int_0^{x_0} \frac{1}{x^{\beta}} dx \geq \frac{\alpha}{x_0} \int_0^{x_0} \frac{e^{-\lambda x}}{x^{\beta}} dx = \frac{1}{x_0} \int_0^{x_0} k(x)dx > 1,
\end{align*}
which implies the inequality
\begin{align*}
x_0 < \bigg( \frac{\alpha}{1 - \beta} \bigg)^{1 / \beta}.
\end{align*}
By \cite[Thm.~6.1.v]{Sato-Yamazato} and (\ref{E-TS-one-sided}), (\ref{Var-TS-one-sided}) we have
\begin{align*}
x_0 > \mathbb{E}[X] - \big( 3 {\rm Var}[X] \big)^{1/2} = \Gamma(1-\beta) \frac{\alpha}{\lambda^{1-\beta}} - \bigg( 3 \Gamma(2-\beta) \frac{\alpha}{\lambda^{2-\beta}} \bigg)^{1/2}.
\end{align*}
Moreover, by \cite[Thm.~6.1.vi]{Sato-Yamazato} we have $x_0 > \xi_0$, where the point $\xi_0 \in (0,\infty)$ is given
\begin{align*}
\xi_0 = \sup\{ u > 0 : k(u) \geq 1 \}
\end{align*}
with $k$ being the strictly decreasing function (\ref{k-one-sided}). Hence, $x$ is the unique solution of the fixed point equation (\ref{fixed-point-mode}). Summing up, we have established relation (\ref{location-one-sided}).
\end{proof}

The unique point $x_0 \in (0,\infty)$ from Theorem~\ref{thm-unimodal-one-sided} is called the \emph{mode} of $\eta$.

\begin{proposition}
We have the asymptotic behaviour
\begin{align*}
\ln g(x) \sim - \frac{1-\beta}{\beta} (\alpha \Gamma(1-\beta))^{\frac{1}{1-\beta}} x^{- \frac{\beta}{1-\beta}} \quad \text{as $x \downarrow 0$.}
\end{align*}
In particular, $g(x) \rightarrow 0$ as $x \downarrow 0$.
\end{proposition}

\begin{proof}
We have $\eta \in {\rm I}_6$ in the sense of \cite[page~275]{Sato-Yamazato} and
\begin{align*}
k(x) \sim \alpha x^{-\beta} \quad \text{as $x \downarrow 0$.}
\end{align*}
Thus, the assertion follows from \cite[Thm.~5.2]{Sato-Yamazato}.
\end{proof}

We shall now investigate the asymptotic behaviour of the densities for large values of $x$. Our idea is to relate the tail of the infinitely divisible distribution to the tail of the L\'{e}vy measure. Results of this kind have been established, e.g., in \cite{Embrechts} and \cite{Watanabe}.

We denote by $\nu := F$ the L\'{e}vy measure of $\eta$, which is given by (\ref{Levy-measure-one-sided}). Then we have
\begin{align*}
\overline{\nu}(r) := \nu((r,\infty)) > 0 \quad \text{for all $r \in \mathbb{R}$,}
\end{align*}
i.e. $\eta \in \mathbb{D}_+$ in the sense of \cite{Watanabe}.
We define the normalized L\'{e}vy measure $\nu_{(1)}$ on $(1,\infty)$ as
\begin{align*}
\nu_{(1)}(dx) := \frac{1}{\nu((1,\infty))} \mathbbm{1}_{(1,\infty)}(x) \nu(dx)
\end{align*}
According to \cite{Watanabe} we say that $\nu_{(1)} \in \mathcal{L}(\gamma)$ for some $\gamma \geq 0$, if for all $a \in \mathbb{R}$ we have
\begin{align*}
\overline{\nu}_{(1)}(r+a) \sim e^{-a \gamma} \overline{\nu}_{(1)}(r) \quad \text{as $r \rightarrow \infty$.}
\end{align*}

\begin{lemma}\label{lemma-nu-1}
We have $\nu_{(1)} \in \mathcal{L}(\lambda)$.
\end{lemma}

\begin{proof}
Let $a \in \mathbb{R}$ be arbitrary. By l'H\^{o}pital's rule we have
\begin{align*}
&\lim_{r \rightarrow \infty} \frac{\overline{\nu}_{(1)}(r+a)}{\overline{\nu}_{(1)}(r)} = \lim_{r \rightarrow \infty} \int_{r+a}^{\infty} \frac{e^{-\lambda x}}{x^{1 + \beta}} dx \bigg/ \int_{r}^{\infty} \frac{e^{-\lambda x}}{x^{1 + \beta}} dx
\\ &= \lim_{r \rightarrow \infty} \frac{e^{-\lambda (r+a)}}{(r+a)^{1 + \beta}} \frac{r^{1 + \beta}}{e^{-\lambda r}} = e^{- a \lambda} \lim_{r \rightarrow \infty} \bigg( \frac{r}{r+a} \bigg)^{1 + \beta} = e^{- a \lambda},
\end{align*}
showing that $\nu_{(1)} \in \mathcal{L}(\lambda)$.
\end{proof}

We define the quantity $d^*$ as
\begin{align*}
d^* := \limsup_{r \rightarrow \infty} \frac{\overline{\nu_{(1)} * \nu_{(1)}}(r)}{\overline{\nu_{(1)}}(r)}.
\end{align*}

\begin{lemma}\label{lemma-d-star-1}
We have the identity
\begin{align*}
d^* = \frac{2 \alpha}{\beta \nu((1,\infty))}.
\end{align*}
\end{lemma}

\begin{proof}
We denote by $g : \mathbb{R} \rightarrow \mathbb{R}$ the density of the normalized L\'{e}vy measure
\begin{align*}
g(x) := \frac{\alpha}{\nu((1,\infty))} \frac{e^{-\lambda x}}{x^{1 + \beta}} \mathbbm{1}_{(1,\infty)}(x).
\end{align*}
Using l'H\^{o}pital's rule, we obtain
\begin{align*}
\lim_{r \rightarrow \infty} \frac{\overline{\nu_{(1)} * \nu_{(1)}}(r)}{\overline{\nu_{(1)}}(r)} &= \lim_{r \rightarrow \infty} \frac{\int_1^r \int_1^{\infty} g(x-y) g(y) dy dx}{\int_1^r g(x)dx} = \lim_{r \rightarrow \infty} \frac{\int_{1}^{\infty} g(r-y)g(y) dy}{g(r)} 
\\ &= \frac{\alpha}{\nu((1,\infty))} \lim_{r \rightarrow \infty} \frac{r^{1 + \beta}}{e^{-\lambda r}} \int_1^{r-1} \frac{e^{-\lambda (r-y)}}{(r-y)^{1 + \beta}} \frac{e^{-\lambda y}}{y^{1 + \beta}} dy
\\ &= \frac{\alpha}{\nu((1,\infty))} \lim_{r \rightarrow \infty} \int_1^{r-1} \bigg( \frac{r}{(r-y)y} \bigg)^{1 + \beta} dy.
\end{align*}
By symmetry, for all $r \in (1,\infty)$ we have
\begin{align*}
\int_1^{r-1} \bigg( \frac{r}{(r-y)y} \bigg)^{1 + \beta} dy &= \int_1^{r/2} \bigg( \frac{r}{(r-y)y} \bigg)^{1 + \beta} dy + \int_{r/2}^{r-1} \bigg( \frac{r}{(r-y)y} \bigg)^{1 + \beta} dy
\\ &= 2 \int_1^{r/2} \bigg( \frac{r}{(r-y)y} \bigg)^{1 + \beta} dy.
\end{align*}
Using the estimate
\begin{align*}
\frac{r}{r-y} \leq 2 \quad \text{for all $r \in (0,\infty)$ and $y \in [1,r/2]$,}
\end{align*}
by Lebesgue's dominated convergence theorem we obtain
\begin{align*}
\lim_{r \rightarrow \infty} \int_1^{\infty} \bigg( \frac{r}{(r-y)y} \bigg)^{1 + \beta} \mathbbm{1}_{[1,r/2]}(y) dy 
&= \int_1^{\infty} \lim_{r \rightarrow \infty} \bigg( \frac{r}{(r-y)y} \bigg)^{1 + \beta} \mathbbm{1}_{[1,r/2]}(y) dy
\\ &= \int_1^{\infty} \frac{1}{y^{1 + \beta}} dy = \frac{1}{\beta},
\end{align*}
which completes the proof.
\end{proof}

For a probability measure $\rho$ on $(\mathbb{R},\mathcal{B}(\mathbb{R}))$ denote by $\hat{\rho}$ the cumulant generating function
\begin{align*}
\hat{\rho}(s) := \int_{\mathbb{R}} e^{sx} \rho(dx).
\end{align*}

\begin{lemma}\label{lemma-d-star-2}
The following identity holds:
\begin{align*}
2 \hat{\nu}_{(1)}(\lambda) =  \frac{2 \alpha}{\beta \nu((1,\infty))}.
\end{align*}
\end{lemma}

\begin{proof}
The computation
\begin{align*}
\hat{\nu}_{(1)}(\lambda) = \int_{\mathbb{R}} e^{\lambda x} \nu_{(1)}(dx) = \frac{\alpha}{\nu((1,\infty))} \int_1^{\infty} \frac{1}{x^{1 + \beta}} dx = \frac{\alpha}{\beta \nu((1,\infty))}
\end{align*}
yields the desired identity.
\end{proof}

\begin{lemma}\label{lemma-C-star}
We have the identity
\begin{align*}
\hat{\eta}(\lambda) = \exp ( - \alpha \Gamma(-\beta) \lambda^{\beta} ).
\end{align*}
\end{lemma}

\begin{proof}
This is a direct consequence of (\ref{cumulant-TS-one-sided}).
\end{proof}

\begin{theorem}\label{thm-asymp-one-sided}
We have the asymptotic behaviour
\begin{align}\label{asymp-final-one-sided}
g(x) \sim \alpha \exp ( - \alpha \Gamma(-\beta) \lambda^{\beta} ) \frac{e^{-\lambda x}}{x^{1 + \beta}} \quad \text{as $x \rightarrow \infty$.}
\end{align}
\end{theorem}

\begin{proof}
By Lemma~\ref{lemma-nu-1} we have $\nu_{(1)} \in \mathcal{L}(\lambda)$ and by Lemmas~\ref{lemma-d-star-1}, \ref{lemma-d-star-2} we have $d^* = 2 \hat{\nu}_{(1)}(\lambda) < \infty$. Thus, \cite[Thm.~2.2.ii]{Watanabe} applies and yields $C_* = C^* = \hat{\eta}(\lambda)$, where
\begin{align*}
C_* := \liminf_{r \rightarrow \infty} \frac{\overline{\eta}(r)}{\overline{\nu}(r)} \quad \text{and} \quad C^* := \limsup_{r \rightarrow \infty} \frac{\overline{\eta}(r)}{\overline{\nu}(r)}.
\end{align*}
Using Lemma~\ref{lemma-C-star} and l'H\^{o}pital's rule we obtain
\begin{align*}
\exp ( - \alpha \Gamma(-\beta) \lambda^{\beta} ) &= \lim_{r \rightarrow \infty} \frac{\overline{\eta}(r)}{\overline{\nu}(r)} = \lim_{r \rightarrow \infty} \int_r^{\infty} g(x) dx \bigg/ \alpha \int_r^{\infty} \frac{e^{-\lambda x}}{x^{1 + \beta}} dx
\\ &= \frac{1}{\alpha} \lim_{r \rightarrow \infty} g(r) \bigg/ \frac{e^{-\lambda r}}{r^{1 + \beta}},
\end{align*}
and hence, we arrive at (\ref{asymp-final-one-sided}).
\end{proof}

Now, we proceed with general two-sided tempered stable distributions. Let
\begin{align*}
\eta = {\rm TS}(\alpha^+,\lambda^+,\beta^+;\alpha^-,\lambda^-,\beta^-)
\end{align*}
be a tempered stable distribution with $\beta^+,\beta^- \in (0,1)$. For bilateral Gamma distributions (i.e. $\beta^+ = \beta^- = 0$) the behaviour of the densities was treated in \cite{Kuechler-Tappe-shapes}.

In view of the L\'{e}vy measure (\ref{Levy-measure-TS}), we can express the characteristic function of $\eta$ as (\ref{cf-self-d}),
where $k : \mathbb{R} \rightarrow \mathbb{R}$ denotes the function
\begin{align*}
k(x) = \alpha^+ \frac{e^{-\lambda^+ x}}{x^{\beta^+}}
\mathbbm{1}_{(0,\infty)}(x) - \alpha^- \frac{e^{- \lambda^-
|x|}}{|x|^{\beta^-}} \mathbbm{1}_{(-\infty,0)}(x), \quad x \in
\mathbb{R}.
\end{align*}
Note that $k \geq 0$ on $(0,\infty)$, that $k \leq 0$ on $(-\infty,0)$ and that $k$ is strictly decreasing on $(0,\infty)$ and $(-\infty,0)$. Furthermore, we have $k(0+) = \infty$, $k(0-) = -\infty$ and $\int_{-1}^1 |k(x)| dx < \infty$.

Arguing as above, $\eta$ is selfdecomposable, absolutely continuous and of class $L$ in the sense of \cite{Sato-Yamazato} with characteristic function being of the form (1.5) in \cite{Sato-Yamazato} with $\gamma_0 = 0$ and $\sigma^2 = 0$. In what follows, we denote by $g$ the density of $\eta$.

Note that $\eta = \eta^+ * \eta^-$ with $\eta^+ = {\rm TS}(\alpha^+,\beta^+,\gamma^+)$ and $\eta^- = \widetilde{\nu}$ with $\nu = {\rm TS}(\alpha^-,\beta^-,\gamma^-)$ and $\widetilde{\nu}$ denoting the dual of $\nu$.

\begin{theorem}\label{thm-unimodal}
The density $g$ of the tempered stable distribution $\eta$ is of class $C^{\infty}(\mathbb{R})$ and there exists a point $x_0 \in \mathbb{R}$ such that (\ref{mode-x0-1})--(\ref{mode-x0-3}) are satisfied. Moreover, we have
$x_0 \in (x_0^-,x_0^+)$, where $x_0^+$ denotes the mode of $\eta^+$ and $x_0^-$ denotes the mode of $\eta^-$.
\end{theorem}

\begin{proof}
Since $k(0+) + |k(0-)| = \infty$, it follows from \cite[Thm.~1.2]{Sato-Yamazato} that $g \in C^{\infty}(\mathbb{R})$. 

Furthermore, since $\int_{-1}^1 |k(x)| dx < \infty$, the distribution $\eta$ is of type ${\rm III}_6$ in the sense of \cite[page~275]{Sato-Yamazato}. According to \cite[Thm.~1.3.xi]{Sato-Yamazato} there exists a point $x_0 \in \mathbb{R}$ with the claimed properties.

Since $|k(0-)| > 1$, an application of \cite[Thm.~4.1.iv]{Sato-Yamazato} yields $x_0 \in (x_0^-,x_0^+)$.
\end{proof}

\begin{remark}
Using Theorem~\ref{thm-unimodal-one-sided}, the mode $x_0$ from Theorem~\ref{thm-unimodal} is located as
\begin{equation}\label{location-mode}
\begin{aligned}
&- \min \bigg\{ \Gamma(1-\beta^-) \frac{\alpha^-}{(\lambda^-)^{1-\beta^-}}, \bigg( \frac{\alpha^-}{1 - \beta^-} \bigg)^{1 / \beta^-} \bigg\} < x_0 
\\ &< \min \bigg\{ \Gamma(1-\beta^+) \frac{\alpha^+}{(\lambda^+)^{1-\beta^+}}, \bigg( \frac{\alpha^+}{1 - \beta^+} \bigg)^{1 / \beta^+} \bigg\}.
\end{aligned}
\end{equation}
\end{remark}

We shall now investigate the asymptotic behaviour of the densities for large values of $x$. By symmetry, it suffices to consider the situation $x \rightarrow \infty$.

\begin{theorem}
Let $g$ be the density of the tempered stable distribution $\eta$. Then we have the asymptotic behaviour
\begin{align}\label{asymp-final}
g(x) \sim C \frac{e^{-\lambda^+ x}}{x^{1 + \beta^+}} \quad \text{as $x \rightarrow \infty$,}
\end{align}
where the constant $C > 0$ is given by
\begin{align*}
C = \alpha^+ \exp \Big( - \alpha^+ \Gamma(-\beta^+) (\lambda^+)^{\beta^+} + \alpha^- \Gamma(-\beta^-) \big[ (\lambda^+ + \lambda^-)^{\beta^-} - (\lambda^-)^{\beta^-} \big] \Big).
\end{align*}
\end{theorem}

\begin{proof}
Let $\nu := F$ be the L\'{e}vy measure of $\eta$ given by (\ref{Levy-measure-TS}). Arguing as in the proofs of Lemma~\ref{lemma-nu-1} and Lemmas~\ref{lemma-d-star-1}, \ref{lemma-d-star-2} we have $\nu_{(1)} \in \mathcal{L}(\lambda^+)$ and $d^* = 2 \hat{\nu}_{(1)}(\gamma) < \infty$. Thus, \cite[Thm.~2.2.ii]{Watanabe} applies and yields $C_* = C^* = \hat{\eta}(\lambda^+)$. By (\ref{cumulant}) we have
\begin{align*}
\hat{\eta}(\lambda^+) &= \int_{\mathbb{R}} e^{\lambda^+ x} \eta(dx) = \exp(\Psi(\lambda^+))
\\ &= \exp \Big( - \alpha^+ \Gamma(-\beta^+) (\lambda^+)^{\beta^+} + \alpha^- \Gamma(-\beta^-) \big[ (\lambda^+ + \lambda^-)^{\beta^-} - (\lambda^-)^{\beta^-} \big] \Big),
\end{align*}
and hence, proceeding as in the proof of Theorem~\ref{thm-asymp-one-sided} we arrive at (\ref{asymp-final}).
\end{proof}

\begin{remark}
We shall now compare the densities of general tempered stable distributions with those of bilateral Gamma distributions ($\beta^+ = \beta^- = 0$):

\begin{itemize}
\item The densities of tempered stable distributions are generally not available in closed form. For bilateral Gamma distributions we have a representation in terms of the Whittaker function, see \cite[Section~3]{Kuechler-Tappe-shapes}.

\item Theorem~\ref{thm-unimodal} and \cite[Thm.~5.1]{Kuechler-Tappe-shapes} show that both, tempered stable and bilateral Gamma distributions, are unimodal.

\item The location of the mode $x_0$ is generally not known, but, due to (\ref{location-mode}) we can determine an interval in which it is located. For some results regarding the mode of bilateral Gamma densities, we refer to \cite[Prop.~5.2]{Kuechler-Tappe-shapes}.

\item According to Theorem~\ref{thm-unimodal}, the densities of tempered stable distributions are of class $C^{\infty}(\mathbb{R})$. This is not true for bilateral Gamma distributions, where the degree of smoothness depends on the parameters $\alpha^+, \alpha^-$. More precisely, we have $g \in C^N(\mathbb{R} \setminus \{0\})$ and $g \in C^{N-1}(\mathbb{R}) \setminus C^N(\mathbb{R})$, where $N \in \mathbb{N}_0$ is the unique integer such that $N < \alpha^+ + \alpha^- \leq N+1$, see \cite[Thm.~4.1]{Kuechler-Tappe-shapes}.

\item For tempered stable distributions the densities have the asymptotic behaviour (\ref{asymp-final}). In contrast to this, the densities of bilateral Gamma distributions have the asymptotic behaviour
\begin{align*}
g(x) \sim C x^{\alpha^+ - 1} e^{-\lambda^+ x} \quad \text{as $x \rightarrow \infty$,}
\end{align*}
for some constant $C > 0$, see \cite[Section~6]{Kuechler-Tappe}.
\end{itemize}
\end{remark}

\section{Density transformations of tempered stable processes}\label{sec-density-transformations}

Equivalent changes of measure are important for applications, e.g. for option pricing in financial mathematics.
The purpose of the present section is to determine all locally equivalent measure changes under which $X$ remains tempered stable. This was already outlined in \cite[Example~9.1]{Cont-Tankov}. Here, we are also interested in the corresponding density process. For the computation of the density process we will use that $X = X^+ - X^-$ can be decomposed as the difference of two independent subordinators.

In order to apply the results from Section~33 in \cite{Sato}, we assume that $\Omega = \mathbb{D}(\mathbb{R}_+)$ is the space of c\`{a}dl\`{a}g functions equipped with the natural filtration $\mathcal{F}_t = \sigma(X_s : s \in [0,t])$ and the $\sigma$-algebra $\mathcal{F} = \sigma(X_t : t \geq 0)$, where $X$ denotes the canonical process $X_t(\omega) = \omega(t)$. Let $\mathbb{P}$ be a probability measure on $(\Omega,\mathcal{F})$ such that
\begin{align*}
X \sim {\rm TS}(\alpha_1^+,\beta_1^+,\lambda_1^+;\alpha_1^-,\beta_1^-,\lambda_1^-) 
\end{align*}
is a tempered stable process. Furthermore, let $\mathbb{Q}$ be another probability measure on $(\Omega,\mathcal{F})$ such that 
\begin{align*}
X \sim {\rm TS}(\alpha_2^+,\beta_2^+,\lambda_2^+;\alpha_2^-,\beta_2^-,\lambda_2^-) 
\end{align*}
under $\mathbb{Q}$. We shall now investigate, under which conditions the probability measures $\mathbb{P}$ and $\mathbb{Q}$ are locally equivalent.

\begin{proposition}\label{prop-TS-eq-measure-change}
The following statements are equivalent:
\begin{enumerate}
\item The measures $\mathbb{P}$ and $\mathbb{Q}$ are locally equivalent.

\item We have $\alpha_1^+ = \alpha_2^+$, $\alpha_1^- = \alpha_2^-$, $\beta_1^+ = \beta_2^+$ and $\beta_1^- = \beta_2^-$.
\end{enumerate}
\end{proposition}

\begin{proof}
The Radon-Nikodym derivative $\Phi = \frac{d F_2}{d F_1}$ of the L\'{e}vy measures is given by
\begin{align*}
\Phi(x) &= \frac{\alpha_2^+}{\alpha_1^+} x^{\beta_1^+ - \beta_1^-} e^{-(\lambda_2^+ - \lambda_1^+)x} \mathbbm{1}_{(0,\infty)}(x) 
\\ &\quad + \frac{\alpha_2^-}{\alpha_1^-} |x|^{\beta_2^+ - \beta_2^-} e^{-(\lambda_2^- - \lambda_1^-)|x|} \mathbbm{1}_{(-\infty,0)}(x), \quad x \in \mathbb{R}.
\end{align*}
According to \cite[Thm.~33.1]{Sato}, the measures $\mathbb{P}$ and $\mathbb{Q}$ are locally equivalent if and only if
\begin{align}\label{cond-Sato}
\int_{\mathbb{R}} \Big( 1 - \sqrt{\Phi(x)} \Big)^2 F_1(dx) < \infty.
\end{align}
Since we have
\begin{align*}
&\int_{\mathbb{R}} \Big( 1 - \sqrt{\Phi(x)} \Big)^2 F_1(dx) 
\\ &= \int_0^{\infty} \bigg( 1 - \sqrt{\frac{\alpha_2^+}{\alpha_1^+}} x^{(\beta_1^+ - \beta_1^-)/2} e^{-(\lambda_2^+ - \lambda_1^+) x/2} \bigg)^2 \frac{\alpha_1^+}{x^{1 + \beta_1^+}} e^{-\lambda_1^+ x} dx
\\ &\quad + \int_{-\infty}^0 \bigg( 1 - \sqrt{\frac{\alpha_2^-}{\alpha_1^-}} |x|^{(\beta_2^+ - \beta_2^-)/2} e^{-(\lambda_2^- - \lambda_1^-)|x|/2} \bigg)^2 \frac{\alpha_2^+}{|x|^{1 + \beta_2^+}} e^{-\lambda_2^+ |x|} dx
\\ &= \int_0^{\infty} \bigg( \sqrt{\alpha_1^+} x^{-(1 + \beta_1^+)/2} e^{-(\lambda_1^+ / 2) x} - \sqrt{\alpha_2^+} x^{-(1 + \beta_2^+)/2} e^{-(\lambda_2^+ / 2 ) x} \bigg)^2 dx
\\ &\quad + \int_0^{\infty} \bigg( \sqrt{\alpha_1^-} x^{-(1 + \beta_1^-)/2} e^{-(\lambda_1^- / 2) x} - \sqrt{\alpha_2^-} x^{-(1 + \beta_2^-)/2} e^{-(\lambda_2^- / 2) x} \bigg)^2 dx,
\end{align*}
condition (\ref{cond-Sato}) is satisfied if and only if we have $\alpha_1^+ = \alpha_2^+$, $\alpha_1^- = \alpha_2^-$, $\beta_1^+ = \beta_2^+$ and $\beta_1^- = \beta_2^-$.
\end{proof}

Now suppose that $\alpha_1^+ = \alpha_2^+ =: \alpha^+$, $\alpha_1^- = \alpha_2^- =: \alpha^-$, $\beta_1^+ = \beta_2^+ =: \beta^+$ and $\beta_1^- = \beta_2^- =: \beta^-$. We shall determine the Radon-Nikodym derivatives $\frac{d \mathbb{Q}}{d \mathbb{P}} |_{\mathcal{F}_t}$ for $t \geq 0$.

We decompose $X = X^+ - X^-$ as the difference of two  independent one-sided tempered stable subordinators and denote by $\Psi^+, \Psi^-$ the respective cumulant generating functions, which can be computed by means of (\ref{cumulant-TS-one-sided}), (\ref{cumulent-Gamma-one-sided}).

\begin{definition}\label{def-bilateral-Esscher}
Let $\theta^+ \in (-\infty,\lambda_1^+)$ and
$\theta^- \in (-\infty,\lambda_1^-)$ be arbitrary. The {\rm bilateral
Esscher transform} is defined by
\begin{align*}
\frac{d \mathbb{P}^{(\theta^+,\theta^-)}}{d
\mathbb{P}}\bigg|_{\mathcal{F}_t} &:= \exp \big( \theta^+ X_t^+ - \Psi^+(\theta^+) t \big) 
\times \exp \big( \theta^- X_t^- - \Psi^-(\theta^-) t \big), \quad t \geq 0.
\end{align*}
\end{definition}

The following result shows that the measure transformation from Proposition~\ref{prop-TS-eq-measure-change} is a bilateral
Esscher transform.

\begin{proposition}
We have $\mathbb{Q} = \mathbb{P}^{(\lambda_1^+ - \lambda_2^+, \lambda_1^- - \lambda_2^-)}$.
\end{proposition}

\begin{proof}
The Radon-Nikodym derivative $\Phi = \frac{d F_2}{d F_1}$ of the L\'{e}vy measures is given by
\begin{align*}
\Phi(x) = \frac{d F_2}{d F_1} = e^{-(\lambda_2^+ - \lambda_1^+)x} \mathbbm{1}_{(0, \infty)}(x) + e^{-(\lambda_2^- - \lambda_1^-)|x|} \mathbbm{1}_{(-\infty, 0)}(x).
\end{align*}
A straightforward calculations shows that
\begin{align*}
\sum_{s \leq t} \ln \Phi(\Delta X_s) &= \sum_{s \leq t} \ln e^{-(\lambda_2^+ - \lambda_1^+) \Delta X_s^+} + \sum_{s \leq t} \ln e^{-(\lambda_2^- -  \lambda_1^-) \Delta X_s^-}
\\ &= (\lambda_1^+ - \lambda_2^+) \sum_{s \leq t} \Delta X_s^+ + (\lambda_1^- - \lambda_2^-) \sum_{s \leq t} \Delta X_s^- 
\\ &= (\lambda_1^+ - \lambda_2^+) X_t^+ + (\lambda_1^- - \lambda_2^-) X_t^-, \quad t \geq 0
\end{align*}
as well as
\begin{align*}
&\int_{\mathbb{R}} ( \Phi(x) - 1 ) F_1(dx) 
\\ &= \int_{\mathbb{R}} \Big( e^{-(\lambda_2^+ - \lambda_1^+)x} \mathbbm{1}_{(0, \infty)}(x) + e^{-(\lambda_2^- - \lambda_1^-)|x|} \mathbbm{1}_{(-\infty, 0)}(x) - 1 \Big) F_1(dx)
\\ &= \alpha^+ \int_0^{\infty} \frac{e^{-\lambda_2^+ x} - e^{-\lambda_1^+ x}}{x^{1 + \beta}} dx + \alpha^- \int_0^{\infty} \frac{e^{-\lambda_2^- x} - e^{-\lambda_1^- x}}{x^{1 + \beta}} dx
\\ &= \int_{\mathbb{R}} (e^{(\lambda_2^+ - \lambda_1^+)x} - 1) F_1^+(dx) + \int_{\mathbb{R}} (e^{(\lambda_2^- - \lambda_1^-)x} - 1) F_1^-(dx)
\\ &= \Psi^+(\lambda_1^+ - \lambda_2^+) + \Psi^-(\lambda_1^- - \lambda_2^-).
\end{align*}
According to \cite[Thm.~33.2]{Sato} the Radon-Nikodym derivatives are given by
\begin{align*}
\frac{d \mathbb{Q}}{d \mathbb{P}} \bigg|_{\mathcal{F}_t} &= \exp \bigg( \sum_{s \leq t} \ln \Phi(\Delta X_s) - t \int_{\mathbb{R}} (\Phi(x) - 1) F_1(dx) \bigg)
\\ &= \exp \Big( (\lambda_1^+ - \lambda_2^+) X_t^+ - \Psi^+(\lambda_1^+ - \lambda_2^+) t \Big) 
\\ &\quad \times \exp \Big( (\lambda_1^- - \lambda_2^-) X_t^- - \Psi^-(\lambda_1^- - \lambda_2^-) t \Big)
\\ &= \frac{d \mathbb{P}^{(\theta^+,\theta^-)}}{d
\mathbb{P}}\bigg|_{\mathcal{F}_t}, \quad t \geq 0
\end{align*}
which finishes the proof.
\end{proof}

\section{The $p$-variation index of a tempered stable process}\label{sec-p-variation}

In this section, we compute the $p$-variation index, which is a measure of the smoothness of the sample paths, for tempered stable processes. The $p$-variation has been investigated in various different contexts, see, e.g., \cite{Blumenthal, Greenwood-1, Greenwood-2, Monroe, BN-She-2003, Corcuera, Jacod-1, Jacod-2, Hein}.

For $t \geq 0$ let $\mathcal{Z}[0,t]$ be set of all decompositions 
\begin{align*}
\Pi = \{ 0 = t_0 < t_1 < \ldots < t_n = t \}
\end{align*}
of the interval $[0,t]$.
For a function $f : \mathbb{R}_+ \rightarrow \mathbb{R}$ we define the \emph{$p$-variation} $V_p(f) : \mathbb{R}_+ \rightarrow \overline{\mathbb{R}}_+$ as
\begin{align*}
V_p(f)_t := \sup_{\Pi \in \mathcal{Z}[0,t]} \sum_{\genfrac{}{}{0pt}{}{t_i \in \Pi}{t_i \neq t}} |f_{t_{i+1}} - f_{t_{i}}|^p, \quad t \geq 0.
\end{align*}
Note that for any $t \geq 0$ the relation $V_p(f)_t < \infty$ implies that $V_q(f)_t < \infty$ for all $q > p$. 

\begin{remark}\label{remark-Cantor}
There exist functions $f : \mathbb{R}_+ \rightarrow \mathbb{R}$ with $V_p(f)_t = \infty$ for all $t > 0$ and all $p > 0$. Indeed, let $f := \mathbbm{1}_{\mathbb{Q}_+}$ and fix an arbitrary $t > 0$. For $n \in \mathbb{N}$ we set $t_i := \frac{i}{n} t$, $i = 0,\ldots,n$, for each $i = 1,\ldots,n$ we choose an irrational number $s_i \in \mathbb{R} \setminus \mathbb{Q}$ with $t_{i-1} < s_i < t_i$ and we define the partition
\begin{align*}
\Pi_n := \{ 0 = t_0 < s_1 < t_1 < \ldots < s_n < t_n = t \}.
\end{align*}
Then, for each $p > 0$ we have
\begin{align*}
\sum_{\genfrac{}{}{0pt}{}{u_i \in \Pi_n}{u_i \neq t}} |f_{u_{i+1}} - f_{u_{i}}|^p = 2n \rightarrow \infty \quad \text{for $n \rightarrow \infty$,}
\end{align*}
showing that $V_p(f)_t = \infty$.
\end{remark}

For a function $f : \mathbb{R}_+ \rightarrow \mathbb{R}$ we define the \emph{$p$-variation index}
\begin{align*}
\gamma(f) := \inf \{ p > 0 : V_p(f)_t < \infty \text{ for all } t \geq 0 \},
\end{align*}
where we agree to set $\inf \emptyset := \infty$. The $p$-variation index is a measure of the smoothness of a function $f$ in the sense that smaller $p$-variation indices mean smoother behaviour of $f$. Here are some examples:
\begin{itemize}
\item For every absolutely continuous function $f$ we have $\gamma(f) \leq 1$;

\item For a Brownian motion $W$ we have $\gamma(W(\omega)) = 2$ for almost all $\omega \in \Omega$;

\item For the function $f := \mathbbm{1}_{\mathbb{Q}_+}$ we have $\gamma(f) = \infty$, see Remark~\ref{remark-Cantor}.
\end{itemize}
In fact, for a L\'{e}vy process $X$ the $p$-variation index $\gamma(X(\omega))$ does not depend on $\omega \in \Omega$. In order to determine $\gamma(X)$, we introduce the
\emph{Blumenthal-Getoor index $\beta(X)$} (which goes back to \cite{Blumenthal}) as
\begin{align*}
\beta(X) := \inf \bigg\{ p > 0 : \int_{-1}^1 |x|^p F(dx) < \infty \bigg\},
\end{align*}
where $F$ denotes the L\'{e}vy measure of $X$. Note that $\beta(X) \leq 2$. According to \cite[Thm.~4.1, 4.2]{Blumenthal} and \cite[Thm.~2]{Monroe} we have $\gamma(X) = \beta(X)$ almost surely.
Moreover, by \cite[Thm.~3.1, 3.3]{Blumenthal} we have almost surely
\begin{align*}
&\lim_{t \rightarrow 0} t^{-1/p} X_t = 0, \quad p > \beta(X)
\\ &\limsup_{t \rightarrow 0} t^{-1/p} |X_t| = \infty, \quad p < \beta(X).
\end{align*}
Now, let $X$ be a tempered stable process
\begin{align*}
X \sim {\rm TS}(\alpha^+,\beta^+,\lambda^+;\alpha^-,\beta^-,\lambda^-). 
\end{align*}

\begin{proposition}
We have $\gamma(X) = \beta(X) = \max\{ \beta^+, \beta^- \}$.
\end{proposition}

\begin{proof}
For any $p > 0$ we have
\begin{align*}
\int_{-1}^1 |x|^p F(dx) &= \alpha^+ \int_0^1 \frac{e^{-\lambda^+ x}}{x^{1 + \beta^+ - p}} dx + \alpha^- \int_{-1}^0 \frac{e^{-\lambda^- |x|}}{|x|^{1 + \beta^- - p}} dx
\\ &= \alpha^+ \int_0^1 \frac{e^{-\lambda^+ x}}{x^{1 + \beta^+ - p}} dx + \alpha^- \int_{0}^1 \frac{e^{-\lambda^- x}}{x^{1 + \beta^- - p}} dx,
\end{align*}
which is finite if and only if $p > \max \{ \beta^+, \beta^- \}$. This shows $\beta(X) = \max\{ \beta^+, \beta^- \}$. Since $\gamma(X) = \beta(X)$, the proof is complete.
\end{proof}

We obtain the following characterization of bilateral Gamma processes within the class of tempered stable processes:
 
\begin{corollary}
We have $\gamma(X) = 0$ if and only if $X$ is a bilateral Gamma process.
\end{corollary}

Consequently, we may regard bilateral Gamma processes as the smoothest class of processes among all tempered stable processes.

\section{Exponential stock models driven by tempered stable processes}\label{sec-finance}

In this section, we will deal with exponential stock price models driven by tempered stable processes. For these models, we will be concerned with choices of equivalent martingale measures and option pricing formulas. For the sake of lucidity, we shall only sketch the main results. The proofs and further details will be provided in a subsequent paper.

In the sequel, we fix $\beta^+,\beta^- \in (0,1)$ and a tempered stable process
\begin{align*}
X \sim {\rm TS}(\alpha^+,\beta^+,\lambda^+;\alpha^-,\beta^-,\lambda^-).
\end{align*}
A \emph{tempered stable stock model} is an exponential L\'{e}vy model of the type
\begin{align*}
\left\{
\begin{array}{rcl}
S_t & = & S_0 e^{X_t}
\\ B_t & = & e^{r t}.
\end{array}
\right.
\end{align*}
The process $S$ is a dividend paying stock with deterministic initial value $S_0 > 0$ and dividend rate $q \geq 0$. Furthermore, $B$ is the bank account with interest rate $r \geq 0$. In what follows, we assume that $r \geq q \geq 0$.
An equivalent probability measure $\mathbb{Q} \sim \mathbb{P}$ is a \emph{local martingale measure} (in short, \emph{martingale measure}), if the discounted stock price process
\begin{align*}
\tilde{S}_t := e^{-(r-q)t} S_t = S_0 e^{X_t - (r-q)t}, \quad t \geq 0
\end{align*}
is a local $\mathbb{Q}$-martingale. 
The existence of a martingale measure $\mathbb{Q} \sim \mathbb{P}$ ensures that the stock market is free of arbitrage, and the price of an European option $\Phi(S_T)$, where $T > 0$ is the time of maturity and $\Phi : \mathbb{R} \rightarrow \mathbb{R}$ the payoff profile, is given by
\begin{align*}
\pi = e^{-rT} \mathbb{E}_{\mathbb{Q}}[\Phi(S_T)].
\end{align*}
One method to obtain equivalent probability measures is to perform the Esscher transform, which was pioneered in \cite{Gerber}.

\begin{definition}
Let $\Theta \in (-\lambda^-,\lambda^+)$ be
arbitrary. The \emph{Esscher transform} $\mathbb{P}^{\Theta}$ is
defined as the locally equivalent probability measure with
likelihood process
\begin{align}\label{density-process}
\Lambda_t(\mathbb{P}^{\Theta},\mathbb{P}) := \frac{d
\mathbb{P}^{\Theta}}{d \mathbb{P}}\bigg|_{\mathcal{F}_t} = e^{\Theta
X_t - \Psi(\Theta) t}, \quad t \geq 0,
\end{align}
where $\Psi$ denotes the cumulant generating function given by
(\ref{cumulant}).
\end{definition}

There exists $\Theta \in (-\lambda^-,\lambda^+)$
such that $\mathbb{P}^{\Theta}$ is a martingale measure if and only if
\begin{align}\label{cond-for-Esscher-lambda}
&\lambda^+ + \lambda^- > 1
\\ \label{cond-domain-temp}
\text{and} \quad &r-q \in (f(-\lambda^-),f(\lambda^+ - 1)].
\end{align}
Here $f : [-\lambda^-,\lambda^+ - 1] \rightarrow \mathbb{R}$ denotes the function
\begin{align*}
f(\Theta) := f^+(\Theta) + f^-(\Theta),
\end{align*}
where we have set
\begin{align*}
f^+(\Theta) &:= \alpha^+ \Gamma(-\beta^+) \big[ (\lambda^+ - \Theta -
1)^{\beta^+} - (\lambda^+ - \Theta)^{\beta^+} \big],
\\ f^-(\Theta) &:= \alpha^-
\Gamma(-\beta^-) \big[ (\lambda^- + \Theta + 1)^{\beta^-} - (\lambda^- +
\Theta)^{\beta^-} \big].
\end{align*}
If conditions (\ref{cond-for-Esscher-lambda}) and (\ref{cond-domain-temp}) are satisfied, then $\Theta$ is the unique solution of the equation
\begin{align*}
f(\Theta) = r-q,
\end{align*}
and under the probability measure $\mathbb{P}^{\Theta}$ we have
\begin{align*}
X \sim {\rm TS}(\alpha^+,\beta^+,\lambda^+ - \Theta;\alpha^-,\beta^-,\lambda^- + \Theta). 
\end{align*}
In contrast to the present situation, for bilateral Gamma stock models ($\beta^+ = \beta^- = 0$) condition (\ref{cond-for-Esscher-lambda}) alone is already sufficient for the existence of an Esscher martingale measure, cf. \cite[Remark~4.4]{Kuechler-Tappe-pricing}.

A more sophisticated method to choose an
equivalent martingale measure $\mathbb{Q} \sim \mathbb{P}$ is to minimize the distance
\begin{align*}
\mathbb{E}[ g (\Lambda_1(\mathbb{Q},\mathbb{P})) ]
\end{align*}
for some strictly convex function $g : (0,\infty) \rightarrow
\mathbb{R}$. Here are some popular choices for the function $g$:
\begin{itemize}
\item For $g(x) = x \ln x$ we call $\mathbb{Q}$ the \emph{minimal entropy martingale measure}. 

\item For $g(x) = x^p$ with $p > 1$ we call $\mathbb{Q}$ the \emph{$p$-optimal martingale measure}.

\item For $p = 2$ we call $\mathbb{Q}$ the \emph{variance-optimal martingale measure}.
\end{itemize}
We shall minimize the relative entropy within the class of tempered stable processes by performing \textit{bilateral} Esscher transforms, cf. Definition~\ref{def-bilateral-Esscher}.

\begin{definition}
Let $\theta^+ \in (-\infty,\lambda^+)$ and
$\theta^- \in (-\infty,\lambda^-)$ be arbitrary. The \emph{bilateral
Esscher transform} $\mathbb{P}^{(\theta^+,\theta^-)}$ is defined as
the locally equivalent probability measure with likelihood process
\begin{align*}
\Lambda_t(\mathbb{P}^{(\theta^+,\theta^-)},\mathbb{P}) := \frac{d
\mathbb{P}^{(\theta^+,\theta^-)}}{d
\mathbb{P}}\bigg|_{\mathcal{F}_t} = e^{\theta^+ X_t^+ -
\Psi^+(\theta^+) t} \cdot e^{\theta^- X_t^- - \Psi^-(\theta^-) t},
\quad t \geq 0.
\end{align*}
\end{definition}

Note that the Esscher transforms $\mathbb{P}^{\Theta}$ are special cases of the just introduced bilateral
Esscher transforms $\mathbb{P}^{(\theta^+,\theta^-)}$. Indeed, we have
\begin{align*}
\mathbb{P}^{\Theta} = \mathbb{P}^{(\Theta,-\Theta)}, \quad \Theta
\in (-\lambda^-,\lambda^+).
\end{align*}
It turns out that martingale measures of the form $\mathbb{P}^{(\theta^+,\theta^-)}$ exist if and only if
\begin{align}\label{ineqn-rq-2}
-\alpha^+ \Gamma(-\beta^+) > r-q.
\end{align}
Then there exist $-\infty \leq \theta_1 < \theta_2 \leq \lambda^+ - 1$ and a continuous, strictly increasing, bijective function $\Phi : (\theta_1,\theta_2) \rightarrow (-\infty,\lambda^-)$ such that $\mathbb{P}^{(\theta,\Phi(\theta))}$ is a martingale measure for each $\theta  \in (\theta_1,\theta_2)$. Moreover, under $\mathbb{P}^{(\theta,\Phi(\theta))}$ we have
\begin{align*}
X \sim {\rm TS}(\alpha^+,\beta^+,\lambda^+ - \theta;\alpha^-,\beta^-,\lambda^- - \Phi(\theta)).
\end{align*}
By~Proposition~\ref{prop-TS-eq-measure-change}, all equivalent measure transformations preserving the class of tempered stable processes are bilateral Esscher transforms. Hence, we introduce the set of parameters
\begin{align*}
\mathcal{M}_{\mathbb{P}} := \{ (\theta^+, \theta^-) \in
(-\infty,\lambda^+) \times (-\infty,\lambda^-) \,|\,
\text{$\mathbb{P}^{(\theta^+,\theta^-)}$ is a martingale measure} \}
\end{align*}
such that the bilateral Esscher transform is a martingale measure.
If condition (\ref{ineqn-rq-2}) is satisfied, then we have
\begin{align}\label{set-MM}
\mathcal{M}_{\mathbb{P}} = \{ (\theta,\Phi(\theta)) \in \mathbb{R}^2
\,|\, \theta \in (\theta_1,\theta_2) \},
\end{align}
and within this class there exist parameters minimizing the relative entropy and the $p$-distances. If condition (\ref{ineqn-rq-2}) is not satisfied, then we have $\mathcal{M}_{\mathbb{P}} = \emptyset$. Hence, in contrast to bilateral Gamma stock models, it can happen that there is no equivalent martingale measure under which $X$ remains a tempered stable process. Moreover, in contrast to bilateral Gamma stock models, the function $\Phi$ is not available in closed form, cf. \cite[Remark~6.7]{Kuechler-Tappe-pricing}.

Concerning the existence of the minimal martingale measure, which has been introduced in \cite{FS} with the motivation of constructing optimal hedging strategies, we obtain a similar result as for bilateral Gamma stock models, cf. \cite[Section~7]{Kuechler-Tappe-pricing}. Namely, assuming that $\lambda^+ \geq 2$ and defining the constant
\begin{align}\label{def-c}
c = c(\alpha^+,\alpha^-,\beta^+,\beta^-,\lambda^+,\lambda^-,r,q) := \frac{\Psi(1) - (r-q)}{\Psi(2) - 2\Psi(1)},
\end{align}
the minimal martingale measure $\hat{\mathbb{P}}$ exists if and only if
\begin{align*}
-1 \leq c \leq 0.
\end{align*}
In this case, under the minimal martingale measure $\hat{\mathbb{P}}$ we have
\begin{align*}
X &\sim {\rm TS}((c+1)\alpha^+,\beta^+,\lambda^+;(c+1)\alpha^-,\beta^-,\lambda^-)
\\ &\quad * {\rm TS}(-c \alpha^+,\beta^+,\lambda^+ - 1;-c \alpha^-,\beta^-,\lambda^- + 1).
\end{align*}
Under each of the martingale measures, which we have presented in this section, we can derive option pricing formulas based on Fourier transform techniques. We fix a strike price $K > 0$ and a maturity date $T > 0$. Denoting by $\varphi_{X_T}$ the characteristic function of $X_T$, for $\lambda^+ > 1$ the price of a call option with these parameters is given by
\begin{align*}
\pi = - \frac{e^{-rT} K}{2 \pi} \int_{i \nu - \infty}^{i \nu + \infty} \bigg( \frac{K}{S_0} \bigg)^{iz} \frac{\varphi_{X_T}(-z)}{z(z-i)} dz,
\end{align*}
where $\nu \in (1,\lambda^+)$ is arbitrary. This follows
from \cite[Section~3.1]{Carr-Madan}, see also \cite[Theorem~3.2]{Lewis}.

Summing up the results of this section, we obtain -- compared to stock price models driven by bilateral Gamma processes -- more restrictive conditions concerning the existence of appropriate martingale measures. This is not surprising, as our investigations of this paper show that, in many respects, the properties of bilateral Gamma distributions differ from those of all other tempered stable distributions.

\section*{Acknowledgement}

The authors are grateful to Gerd Christoph for discussions about the optimal constant in the Berry-Esseen theorem.

The authors are also grateful to the Associate Editor and two anonymous reviewers for valuable comments and suggestions.

\end{document}